\documentclass{amsart}

\usepackage{indentfirst}
\usepackage{amssymb}
\usepackage{amsmath}
\usepackage{graphics}
\usepackage{epsfig}
\usepackage{xcolor}

\usepackage[utf8]{inputenc}
\usepackage{latexsym}
\usepackage{graphicx}

\catcode`@=11 \@addtoreset{equation}{section}
\renewcommand\theequation{\thesection.\@arabic\c@equation}
\catcode`@=12

\newcommand{\RR}{\mathbb{R}}

\newcommand{\Ez}{{\mathcal{E}}_{\zz}}

\expandafter\chardef\csname pre amssym.def
at\endcsname=\the\catcode`\@ \catcode`\@=11
\def\undefine#1{\let#1\undefined}
\def\newsymbol#1#2#3#4#5{\let\next@\relax
 \ifnum#2=\@ne\let\next@\msafam@\else
 \ifnum#2=\tw@\let\next@\msbfam@\fi\fi
 \mathchardef#1="#3\next@#4#5}
\def\mathhexbox@#1#2#3{\relax
 \ifmmode\mathpalette{}{\m@th\mathchar"#1#2#3}%
 \else\leavevmode\hbox{$\m@th\mathchar"#1#2#3$}\fi}
\def\hexnumber@#1{\ifcase#1 0\or 1\or 2\or 3\or 4\or 5\or 6\or 7\or 8\or
 9\or A\or B\or C\or D\or E\or F\fi}

\font\teneufm=eufm10 \font\seveneufm=eufm7 \font\fiveeufm=eufm5
\newfam\eufmfam
\textfont\eufmfam=\teneufm \scriptfont\eufmfam=\seveneufm
\scriptscriptfont\eufmfam=\fiveeufm

\catcode`\@=\csname pre amssym.def at\endcsname

\newcommand{\eqn}{\begin{eqnarray}}
\newcommand{\een}{\end{eqnarray}}

\newtheorem {Theorem}  {Theorem}

\numberwithin{Theorem}{section}
\newtheorem{Lemma}[Theorem]{Lemma}
\newtheorem{Claim}[Theorem]{Claim}
\newtheorem{Proposition}[Theorem]{Proposition}

\theoremstyle{Remark}
\newtheorem{Remark}[Theorem]{Remark}

\newtheorem{Corollary}[Theorem]{Corollary}

\newcommand{\xx}{\mbox{\boldmath $x$}}

 \newcommand{\ww}{{\bf w}}
\newcommand{\uu}{\mbox{\boldmath $u$}}
\newcommand{\zz}{\mbox{\boldmath $z$}}

\newcommand{\HH}{\mbox{\boldmath $H$}}

\begin{document}

\title[Micropolar fluids with nonlinear damping]{
Improved decay results for micropolar flows with nonlinear damping
}


\author[C. F. Perusato]{Cilon F. Perusato}
\address[C. F. Perusato]{Departamento de Matem\'atica. Universidade Federal de Pernambuco, 
 CEP 50740-560, Recife - PE. Brazil}
\email{cilon.perusato@ufpe.br}

\author[F. D. Vega]{Franco D. Vega}
\address[F. D. Vega]{Departamento de Matem\'atica. Universidade Federal de Pernambuco, 
	CEP 50740-560, Recife - PE. Brazil}
\email{franco.diaz@ufpe.br}

\thanks{C. F. Perusato was partially supported by Cnpq through grant $\#$ 310444/2022 - 5, bolsa PQ} 



\date{\today}

%

%

\begin{abstract}
We examine the long-time behavior of solutions (and their derivatives) to the micropolar equations with nonlinear velocity damping. Additionally, we get a speed-up gain of $ t^{1/2} $ for the angular velocity, consistent with established findings for classic micropolar flows lacking nonlinear damping. Consequently, we also obtain a sharper result regarding the asymptotic stability of the micro-rotational velocity $\ww(\cdot,t)$. Related results of independent interest are also included.
\end{abstract}

\maketitle

\section{Introduction}

In this work we investigate the large time behavior of solutions to the following micropolar equations with nonlinear velocity damping in $\RR^3$.
\begin{equation}\label{micropolar}
	\begin{split}
		\mbox{\boldmath $u$}_t
		+
		\mbox{\boldmath $u$} \cdot \nabla \mbox{\boldmath $u$}
		+
		\nabla \:\!{ \sf{p}}
		=  
		(\:\!\mu + \chi\:\!) \, \Delta \mbox{\boldmath $u$}
		\,+\,2\,
		\chi \, \nabla \times {\bf w} - \eta\, |\uu|^{\beta -1}\uu , \\
		{\bf w}_t
		+
		\mbox{\boldmath $u$} \cdot \nabla \mbox{\bf w}
		= 
		\gamma \;\! \, \Delta \mbox{\bf w}
		\,+\,\kappa\, \nabla \:\!(\:\!\nabla \cdot\;{\bf w} \:\!)
		\,+\,2\,
		\chi \, \nabla \times\, \mbox{\boldmath $u$}
		\,-\, 4 \, \chi \, {\bf w}, \\
		\nabla \cdot \mbox{\boldmath $u$}(\cdot,t)   =   0,
		\\
		(\uu,\ww)(\cdot, 0) = (\uu_0,\ww_0) \in \!\;\!\mbox{\boldmath $L$}^{2}_{\sigma}(\mathbb{R}^{3}) \times \!\;\!\mbox{\boldmath $L$}^{2}(\mathbb{R}^{3})  ,
	\end{split}
\end{equation}
where the coefficients 
$ \mu $ (kinematic viscosity), 
$ \gamma $ (angular viscosity), 
 $ \chi $ (vortex or micro-rotation viscosity) and $ \beta\geq 1,\,\eta $ (damping coefficients)
are positive, 
and $ \kappa $ (gyroviscosity) is nonnega\-tive, 
all assumed to be constant.  The functions 
\mbox{$ \mbox{\boldmath $u$} = \mbox{\boldmath $u$}(\xx,t) $}, $\ww = \ww(\xx,t) $, 
and $ {\sf p} = {\sf p}(\xx,t) $
are the flow velocity, micro-rotational velocity and the total pressure, respectively, for $ t>0 $ and $ x \in \RR^n $.
Here,
$ \!\;\!\mbox{\boldmath $L$}^{2}_{\sigma}(\mathbb{R}^{n}) $
denotes the
space
of solenoidal fields
$ \:\!\mbox{\bf v} = (v_{1}, v_{2}, \ldots, v_{n}) \!\:\!\in
\mbox{\boldmath $L$}^{2}(\mathbb{R}^{n}) \equiv L^{2}(\mathbb{R}^{n})^{n} \!\:\!$
with
$ \nabla \!\cdot \mbox{\bf v} \!\;\!= 0 $
in the distributional sense.

The so-called Leray-Hopf (or simply Leray) solutions in $ \mathbb{R}^3 \!$ are 
global mappings (see \cite{Leray1934}) 
$$(\mbox{\boldmath $u$}, \mbox{\bf w})(\cdot,t)
\in C_{\sf w}([\;\!0, \infty), L^{2}_{\sigma}(\mathbb{R}^{3}) 
\hspace{-0.010cm} \times \hspace{-0.020cm} L^{2}(\mathbb{R}^{3})) 
 \!\;\!\cap\hspace{+0.075cm} L^{2}((0, \infty), 
\dot{H}^{\!\;\!1}(\mathbb{R}^{3}) 
\hspace{-0.020cm} \times \hspace{-0.025cm}
\dot{H}^{\!\;\!1}(\mathbb{R}^{3})) $$
that satisfy the equations in weak sense for $ t > 0 $
and in addition
the energy estimate

\begin{equation}\label{eqn_energy_inequality}
\,\,	\|\!\;\!\;\!(\mbox{\boldmath $u$},\mbox{\bf w})(\cdot,t)\!\;\!\;\!
	\|_{\mbox{}_{\scriptstyle \!\:\!L^{\!\;\!2}}}^{\:\!2} 
	\!\;\!+\;\! 
	2 \! \int_{\!\:\!s}^{\:\!t}  \mu \;\! 
	\|\!\;\!\;\!D\mbox{\boldmath $u$}(\cdot,\tau)\!\;\!\;\!
	\|_{\mbox{}_{\scriptstyle \!\:\!L^{\!\;\!2}}}^{\:\!2} 
+
	\gamma \;\! 
	\|\!\;\!\;\!D\mbox{\bf w}(\cdot,\tau)\!\;\!\;\!
	\|_{\mbox{}_{\scriptstyle \!\:\!L^{\!\;\!2}}}^{\:\!2}  
	 d\tau
	\;\!\leq\,
	\|\!\;\!\;\!(\mbox{\boldmath $u$},\mbox{\bf w})(\cdot,s)\!\;\!\;\!
	\|_{\mbox{}_{\scriptstyle \!\:\!L^{\!\;\!2}}}^{\:\!2} 
\end{equation}
\mbox{} \vspace{-0.150cm} \\
for all $ t > s $,
for $ s = 0 $
and almost every $ s > 0 $. The presence of the nonlinear damping term $ \eta\,|\uu|^{\beta-1}\,\uu $ is clearly beneficial to the regularity of weak solutions as we will see. Due to this, one has the following extra regularity property for the velocity field: $ \uu(\cdot, t) \in L^{\beta+1}(\,(0,\infty), L^{\beta+1} (\RR^3)\, ) $, see e.g \cite{CaiJiu2008}. In fact, when $ \beta > 3  $  (or when $ \beta = 3 $ and $4\,\eta\,(\mu + \chi) > 1$) and the initial data are also in $ {\bf H}^1_\sigma \times {\bf{H}}^1 $ there exists a unique global strong solution of \eqref{micropolar}, see Z. Ye \cite{Ye2018}. In order to have the same result, when $ 1\leq\beta <3 $, it is required (as usual) some smallness condition on the initial data, as pointed out by W. Wang and Y. Long \cite{WangLong2021} .     

Let us now give some physical motivation regarding  \eqref{micropolar}. In the 1960s, C.\;Eringen \cite{Eringen1966} 
introduced the micropolar fluids model 
(see the equations \eqref{micropolar} with $ \eta =0 $ above). 
It deals with a class of fluids 
which exhibit certain microscopic effects 
arising from micro-motions and local structures 
of the fluid particles that lead to a
nonsymmetric stress tensor, 
and which are often called polar fluids. 
It includes, as a particular case, 
the well-established Navier-Stokes model~\cite{Lukaszewicz1999}. When one also considers the nonlinear damping term several applications can be taken into account in view of the fact that a system with nonlinear damping arises from the resistance to the motion of the flows. For that reason it describes various physical situations such as porous media flow, drag or friction effects, and some dissipative mechanisms. So, this plays an important role to balance the convection term. 

The existence of Leray solutions (originally constructed for the Navier-Stokes system in \cite{Leray1934}) for the equations \eqref{micropolar}
and other similar dissipative systems is well known, but their uniqueness and exact
regularity properties are still open,\!\footnote{%
	%
	%
	Recent advances seem to give some indication
	that 3D Leray solutions may fail to be unique
	(and therefore smooth), see
	\cite{Albritton2022, BuckmasterVicol2019, JiaSverak2015}
	and references therein.
}
except for small initial data
in suitable spaces \cite{MR1666509, MR0467030, %
	Lukaszewicz1999, MR2778615} or in case of nonlinear damping as mentioned before, when $ \beta >3 $.  Recently, interesting results regarding the regularity of solutions to the micropolar equation without nonlinear damping (i.e.,  when $ \eta = 0  $) was studied separately for $\uu $ and $\ww$ by using a new notion of partial suitable solutions, see \cite{Chamorro, MR4523469}. 

	\subsection{Organization of the paper}
	This work is organized as follows. In subsection 1.2 below, we shall delineate the primary contributions elucidated in this paper.  In Section 2 some preliminary results are provided to prepare the way for the derivation of the main theorems described  thereafter. Although much of this material is essentially known, with the exception of Theorem \ref{inverse-wiegner}, which is novel, several proofs are newly presented and a few improvements are offered. In Section 3, we present the proofs of our main results. This section is divided into two parts: the first part is dedicated to proving Theorem \ref{thm_L2}, while the second part is dedicated to providing the proof of Theorem \ref{thm_inequalities}. An appendix supplements the discussion by providing the proof of a claim necessary to establish Theorem \ref{inverse-wiegner}.    

\subsection{Main results}
We improve some previous results (originally obtained by H. Li and Y. Xiao in \cite{MR4357379}). Namely, we provide a faster decay for the $ L^2 $ norm of the micro-rotational field. We also obtain a better range for the nonlinear damping coefficient $ \beta $, see c.f. Theorem \ref{thm_L2} and remark \ref{rmk1} below. For the Navier-Stokes flows with nonlinear damping, the large time decay was studied by Y. Jia, X. Zhang and B-Q. Dong in \cite{MR2781892}. That being mentioned, our primary result can be stated as follows.

\begin{Theorem}\label{thm_L2}
		Let $(\uu_0,\ww_0) \in {\bf L}^2_\sigma(\RR^3) \times {\bf L}^2 (\RR^3)$ and $ (\uu, \ww)(\cdot,t) $ any Leray solution to  \eqref{micropolar}. Then, 
		\begin{equation}\label{eqn_u_thm1}
			 \|\uu(\cdot,t) \|_{L^2(\RR^n)} \to  0 \quad \text{as}\,\,t\to\infty, \quad \forall \beta \geq 7/3. 
		\end{equation}
	Moreover, for the micro-rotational  field $ \ww(\cdot, t) $, one has
		\begin{equation}\label{eqn_w_thm1}
		t^{1/2}\|\ww(\cdot,t) \|_{L^2(\RR^n)} \to  0 \quad \text{as}\,\,t\to\infty, \quad \forall \beta \geq 1. 
	\end{equation}
\end{Theorem}
Now, let us delve into the results concerning the $\dot{H}^m$ norms of the solutions. For convenience, we define $ \lambda_0(\alpha):=\limsup_{t \to \infty}\,t^\alpha\,\|\uu(\cdot, t) \|_{L^2(\RR^n)} $, for some constant $ \alpha \geq 0 $. We assume that 
\begin{equation}\label{eqn:lambda-assumption}
	\lambda_0(\alpha) < \infty,
\end{equation}
$ \text{for some }\,\,\alpha \geq 0$.  

\begin{Remark}[Reasonability of the assumption \eqref{eqn:lambda-assumption}]\label{rmk1}
	Theorem \ref{thm_L2} can be formulated in terms of the decay character, see \cite{MR2493562, MR3493117, MR3355116}. Moreover, by using the so-called Fourier Splitting technique developed by M. Schonbek in \cite{MR775190} , the authors in \cite{MR4357379} show that when the initial data are in $L^1 \cap L^2$ the assumption \eqref{eqn:lambda-assumption} 
	is valid with $\alpha = 3/4$, for all $ \beta > \frac{14}{5} $. Actually, adapting the arguments presented in \cite{MR4357379} and here in our  paper it is possible to obtain the same result for all $ \beta > 8/3 $, but this is out of the scope of this work.      
	All of these facts ensure that the assumption \eqref{eqn:lambda-assumption} above holds for some values of $ \alpha \geq 0 $.
\end{Remark}
We can finally state the result for all derivatives of order $ m $ which is the following asymptotic inequalities in $\dot{H}^m$.
\begin{Theorem}\label{thm_inequalities}
Let $(\uu_0,\ww_0) \in {\bf L}^2_\sigma(\RR^3) \times {\bf L}^2 (\RR^3)$ and $ (\uu, \ww)(\cdot,t) $ any Leray solution to  \eqref{micropolar}. Then, 
	\begin{subequations}\label{inequality}
	\begin{equation}\label{inequality_z}
		\limsup_{t \to \infty}\, t^{\alpha+\frac{m}{2}} \|D^m \uu(\cdot,t)\|_{L^2(\RR^3)} \leq  C_{\alpha,m}\;\; \lambda_0(\alpha)
	\end{equation}
	and, for the micro-rotational field $\mathbf{w}$, 
	\begin{equation}\label{inequality_w}
		\limsup_{t \to \infty} t^{\alpha+\frac{m+1}{2}} \|D^m \ww(\cdot,t)\|_{L^2(\RR^3)} \leq  \,\frac{{C}_{\alpha,m}}{4\,\chi} \;\; \lambda_0(\alpha),
	\end{equation}	
\end{subequations}
for each $ m \geq 2 $ and $ \beta \geq \frac{4\,\alpha + 7}{4\,\alpha + 3} $, where $$ C_{\alpha,m}= 2^{\alpha + m/2}\,(\mu + \chi)^{-\frac{m+1}{2}}.  $$
\end{Theorem}
The case $m = 1$ is addressed in Lemma \ref{estimativa grad un medio} and in Proposition \ref{proposition}.  The last theorem allows us to obtain lower bounds for the $L^2$ norm of solutions, see e.g. \cite{MR4546674}. These estimates are closely related to the important results provided by M. Oliver and E. Titi for the Navier-Stokes equations in \cite{MR1749867}.

As a corollary of Theorem \ref{thm_inequalities}, one can also obtain the following sharper stability result. To obtain the following corollary, it suffices to apply the triangle inequality and employ Theorem \ref{thm_inequalities}. 
\begin{Corollary}		
			Let $(\uu_0,\ww_0) \in {\bf L}^2_\sigma(\RR^3) \times {\bf L}^2 (\RR^3)$ and $ (\uu, \ww)(\cdot,t) $ any Leray solution to  \eqref{micropolar}. $If \|(\uu,\ww)(\cdot,t) \| = O(t^{-\alpha})$, then
			\begin{equation*}
				\|D^m\uu(\cdot,t) - D^m\tilde{\uu}(\cdot,t) \|_{L^2} =  O(t^{-\alpha - m/2}) 
			\end{equation*}
		and 
			\begin{equation*}
				\|D^m\ww(\cdot,t) - D^m\tilde{\ww}(\cdot,t) \|_{L^2} =  O(t^{-\alpha - m/2- 1/2}), .
			\end{equation*}
for all	$ \beta \geq \frac{4\,\alpha + 7}{4\,\alpha + 3} $ and each $ m \geq 2 $, where $ (\tilde{\uu},\tilde{\ww}) $ stands for a solution to the problem \eqref{micropolar} with a perturbed initial condition 
$ \left.(\tilde{\uu}, \tilde{\ww})\right|_{t=0}=\left(\uu_{0}+a, \ww_{0}+b\right) $ and $ a,b \in {\bf L}^2_\sigma(\RR^3) \times {\bf L}^2 (\RR^3)$.  		
\end{Corollary}
{\bf Notation.} As usual, ${ \overset{.}{H}^{1}(\mathbb{R}^{n}) } = \overset{.}{H}^{1}(\mathbb{R}^{n})^{n} $ where
$\overset{.}{H}^{1}(\mathbb{R}^{n})$ denotes the homogeneous Sobolev space of order 1 , $e^{\nu \Delta t}$ denotes the heat semigroup. As shown above, boldface letters are used for vector quantities, as in $\uu(x,t)=(u_{1}(x,t),u_{2}(x,t),u_{3}(x,t))$ denotes the field velocity. Also
$\nabla P \equiv \nabla P(\cdot,t) $ denotes the spatial gradient of $P(\cdot,t)$ , $D_{j}=\frac{\partial }{\partial x_{j}}$ , $\nabla \cdot \uu \:=\: D_{1}u_{1} \:+\: D_{2}u_{2} \:+\: D_{3}u_{3}$ is the (spatial) divergence of $\uu(\cdot,t)$ , similarly $\uu \cdot \nabla \uu = u_{1} D_{1}\uu  + ...+ u_{n} D_{n}\uu$ . $\left | \cdot \right |_{2}$ denotes the Euclidean norm in $\mathbb{R}^{3}$ , $\left | \cdot \right |$ denotes the $ \ell^1 $ norm in $\mathbb{R}^{3}$ and $\left \| \cdot \right \|_{L^{q}(\mathbb{R}^{3})} \:\:\:\: ;\:\:\:\:  1\:\leq \:q\: \leq \:\infty$ ,are the standard norms of the Lebesgue spaces $L^{q}(\mathbb{R}^{3})$, with the vector counterparts.

\begin{equation}\label{18}
	\left \| \uu(\cdot,t) \right \|_{L^{q}(\mathbb{R}^{n})}\:=\: \left \{ \sum_{i=1}^{n}\int _{\mathbb{R}^{n}} \left | u_{i}(x,t) \right |^{q}dx\right \}^{\frac{1}{q}}
\end{equation}

\begin{equation}\label{19}  
	\left \| D\uu(\cdot,t) \right \|_{L^{q}(\mathbb{R}^{n})}\:=\: \left \{ \sum_{i,j=1}^{n}\int _{\mathbb{R}^{n}} \left | D_{j}u_{i}(x,t) \right |^{q}dx\right \}^{\frac{1}{q}}
\end{equation}
and, in general,
\begin{equation}
	\left \| D^{m}\uu(\cdot,t) \right \|_{L^{q}(\mathbb{R}^{n})}\:=\: \left \{ \sum_{i,j_{1},...,j_{m}=1}^{n}\int _{\mathbb{R}^{n}} \left | D_{j_{1}}...D_{j_{m}}u_{i}(x,t) \right |^{q}dx\right \}^{\frac{1}{q}}
\end{equation}
if $1\: \leq \: q \:<\: \infty$.
When, $q = \infty$,
\begin{equation}
	\left \| \uu(\cdot,t) \right \|_{L^{\infty}(\mathbb{R}^{n})}\:=\: \max \left \{ \left \| u_{i}(\cdot,t) \right \|_{L^{\infty}(\mathbb{R}^{n})}\:\:\:  :  \:\:\: 1\leq i \leq n  \right \}
\end{equation}
and, for general $m\: \geq \:1$ :
\begin{equation}
	\left \| D^{m}\uu(\cdot,t) \right \|_{L^{\infty}(\mathbb{R}^{n})}\:=\: \max \left \{ \left \| D_{j_{1}}...D_{j_{m}}u_{i}(\cdot,t) \right \|_{L^{\infty}(\mathbb{R}^{n})}\:\:\:  :  \:\:\: 1\leq i,j_{1},...,j_{m} \leq n  \right \}.
\end{equation}

We also defined for simplicity the following norms for $(u,w)$ as usually made in the literature:

\begin{equation}
	\left \| (\uu,\ww) \right \|_{L^{q}(\mathbb{R}^{n})}^{q}\: :=\: \left \| \uu \right \|_{L^{q}(\mathbb{R}^{n})}^{q} \:+\: \left \| \ww \right \|_{L^{q}(\mathbb{R}^{n})}^{q}
\end{equation}
and more generally, for all integers $m\: \geq \: 1$,
\begin{equation}
	\left \| (D^{m}\uu,D^{m}\ww) \right \|_{L^{q}(\mathbb{R}^{n})}^{q}\: :=\: \left \| D^{m}\uu \right \|_{L^{q}(\mathbb{R}^{n})}^{q} \:+\: \left \| D^{m}\ww \right \|_{L^{q}(\mathbb{R}^{n})}^{q}
\end{equation}
for all $1\: \leq q < \infty$ and when $q=\infty$,
\begin{equation}
	\left \| (\uu,\ww) \right \|_{L^{\infty}(\mathbb{R}^{n})}\:=\: \max \left \{ \left \| \uu \right \|_{L^{\infty}(\mathbb{R}^{n})} , \left \| \ww \right \|_{L^{\infty}(\mathbb{R}^{n})} \right \}.
\end{equation}

\section{Some mathematical preliminaries}
Let us recall some basic facts regarding the weak solutions of \eqref{micropolar}. Exactly as in the Navier-Stokes case, it is not known that Leray solutions to the problem \eqref{micropolar} are regular and smooth for all $t>0$. However, it is known that they do behave nicely for all, $t>0$ sufficiently large \cite{Lukaszewicz1999,Leray1934}, say $t>t_{*}$, with
\begin{subequations}\label{regularity_time}
\begin{equation}
	(\uu,\ww)(\cdot,t) \in  C^{\infty }(\mathbb{R}^{3}\times (t_{*},\infty )) 
\end{equation}
and, for each $m \in \mathbb{Z}_{+} $ :
\begin{equation}
	(\uu,\ww)(\cdot,t) \in  C^{0 }( \left [  t_{*},\infty  \right )  , \HH^{m} (\mathbb{R}^{3}) ) 
\end{equation}
\end{subequations}
and such that the elementary ({\it strong}) energy inequality
\begin{align*}
	&\left \| (\uu,\ww)(\cdot,t) \right \|_{L^{2}(\mathbb{R}^{3})}^{2} \:+\:  2\lambda \int_{t_{0}}^{t}\left \| (D\uu,D\ww)(\cdot,\tau ) \right \|_{L^{2}(\mathbb{R}^{3})}^{2}d\tau \\
&+\: 2\kappa \int_{t_{0}}^{t}\left \| \nabla \cdot \ww(\cdot,\tau ) \right \|_{L^{2}(\mathbb{R}^{3})}^{2}d\tau  \:+\: 2\eta  \int_{t_{0}}^{t}\left \| \uu(\cdot,\tau ) \right \|_{L^{\beta +1}(\mathbb{R}^{3})}^{\beta +1}d\tau  \: \leq \: \left \| (\uu,\ww)(\cdot,t_{0}) \right \|_{L^{2}(\mathbb{R}^{3})}^{2}, 
\end{align*}
holds for all $ t >t_0 $ and for a.e. $t_{0} \geq 0$ (including $t_0 = 0$), where $\lambda = \min \left \{ \mu ,\gamma  \right \}$. 

We state now some useful well-known lemmas that are necessary to proof our main results.

\begin{Lemma}\label{GSN}
	For any $ f \in H^2(\RR^3) $, we have
	\begin{align*}
		\left \| f \right \|_{L^{\infty }(\mathbb{R}^{3})}&\left \| \nabla f \right \|_{L^{2}(\mathbb{R}^{3})} 
		\leq C \left \| f \right \|_{L^{2}(\mathbb{R}^{3})}^{\frac{1}{2}} \left \| \nabla f \right \|_{L^{2}(\mathbb{R}^{3})}^{\frac{1}{2}} \left \| \nabla^{2} f \right \|_{L^{2}(\mathbb{R}^{3})}. 
	\end{align*}
\end{Lemma}

\begin{Lemma}\label{lema D alfa}
	Let $\uu\in L^{r}(\mathbb{R}^{n})$ and $e^{\nu \Delta \tau}$ the heat Kernel, then
	\begin{align*}
		\left \| D^{\alpha }\left [ e^{\nu \Delta \tau}\uu \right ] \right \|_{L^{2}(\mathbb{R}^{n})} \: \leq \: K(n,m)\left \| \uu \right \|_{L^{r}(\mathbb{R}^{n})}(\nu \tau )^{-\frac{n}{2}\left ( \frac{1}{r}-\frac{1}{2} \right )-\frac{\left | \alpha  \right |}{2}}
	\end{align*}
	for all $\tau >0$ and $\alpha$ (multi-index), $1\: \leq \: r\: \leq \: 2$, $n\: \geq \:1$, and $m\:=\: \left | \alpha  \right |$. 
\end{Lemma}
The following asymptotic result is a key property for the gradient and it is very important to provide general time decay estimates for solutions to the system \eqref{micropolar} and to other diffusive equations (see \cite{MR4546674} for the general approach).    
\begin{Lemma}\label{estimativa grad un medio}
	For $(\uu,\ww)$	Leray solutions of \eqref{micropolar}, one has
	\begin{equation}\label{t un medio grad}
		\lim _{t\rightarrow +\infty }t^{\frac{1}{2}}\left \| (D\uu,D\ww)(\cdot,t) \right \|_{L^{2}(\mathbb{R}^{3})} \:=\: 0 .
	\end{equation}
\end{Lemma}
\begin{proof}
	The next argument is adapted from \cite{MR1994780}. Define, for simplicity, $\zz(\cdot,t) := (\uu,\ww)(\cdot,t)$ and $D^{m}\zz =(D^{m}\uu,D^{m}\ww)$ , for each $m \geq 0$ integer. In order to show \eqref{t un medio grad}, we used \eqref{micropolar} to get, after a few computations,
	\begin{equation}\label{desi gsn}
		\begin{split}
			&\left \| (D\uu,D\ww)(\cdot,t) \right \|_{L^{2}(\mathbb{R}^{3})}^{2} \:+\: 2 \min \left \{ \mu,\gamma \right \} \int_{t_{0}}^{t}\left \| (D^{2}\uu,D^{2}\ww)(\cdot,\tau) \right \|_{L^{2}(\mathbb{R}^{3})}^{2}d\tau \\
			&\: \leq \: \left \| (D\uu,D\ww)(\cdot,t_{0}) \right \|_{L^{2}(\mathbb{R}^{3})}^{2}\\
			& \quad + \: C\: \int_{t_{0}}^{t} \left \| ( \uu, \ww) \right \|_{L^{2}(\mathbb{R}^{3})}^{\frac{1}{2}} \left \| (\nabla \uu,\nabla \ww) \right \|_{L^{2}(\mathbb{R}^{3})}^{\frac{1}{2}} \left \| (\nabla^{2} \uu,\nabla^{2} \ww) \right \|_{L^{2}(\mathbb{R}^{3})}^{2}d \tau ,
		\end{split}
	\end{equation}
	where we have used a Lemma \ref{GSN}. By \eqref{eqn_energy_inequality}, we can choose $t_{0} \geq t_{*}$ large enough such that
	\begin{equation*}
		C^{2} \left \| ( \uu_{0}, \ww_{0}) \right \|_{L^{2}(\mathbb{R}^{3})} \left \| (D \uu, D \ww)(\cdot,t_{0}) \right \|_{L^{2}(\mathbb{R}^{3})} < \left ( \min \left \{ \mu,\gamma \right \} \right )^{2},
	\end{equation*}
	so that \eqref{desi gsn} gives $\left \| (D \uu, D \ww)(\cdot,t) \right \|_{L^{2}(\mathbb{R}^{3})} \leq \left \| (D \uu, D \ww)(\cdot,t_{0}) \right \|_{L^{2}(\mathbb{R}^{3})}$ for all t near $t_{0}$ by continuity. Actually, with this choice, it follows from \eqref{desi gsn} again that
	\begin{equation*}
		C^{2} \left \| ( \uu_{0}, \ww_{0}) \right \|_{L^{2}(\mathbb{R}^{3})} \left \| (D \uu, D \ww)(\cdot,s) \right \|_{L^{2}(\mathbb{R}^{3})} < \left ( \min \left \{ \mu,\gamma \right \} \right )^{2}, \quad \forall s \geq t_{0}.
	\end{equation*}
	Recalling \eqref{desi gsn}, this implies that
	\begin{equation*}
		\left \| (D \uu, D \ww)(\cdot,t) \right \|_{L^{2}(\mathbb{R}^{3})} \leq \left \| (D \uu, D \ww)(\cdot,t_{0}) \right \|_{L^{2}(\mathbb{R}^{3})} ,
	\end{equation*}
	for all $t \geq t_{0}$. Because a monotonic function $f \in C^{0}(a,\infty) \cap L^{1}(a,\infty)$ has to satisfy $f(t)= O(1/t)$ as $t \rightarrow \infty$, we have
	\begin{equation*}
		\lim _{t\rightarrow \infty }t \left \| (D \uu, D \ww)(\cdot,t) \right \|_{L^{2}(\mathbb{R}^{3})}^{2} \:=\: 0 .
	\end{equation*}

\end{proof}

An important direct consequence of lemma \ref{estimativa grad un medio} is the following proposition.
\begin{Proposition}\label{proposition} 
	Let $(\uu,\ww)(\cdot, t)$ be any Leray  solution of problem \eqref{micropolar} with $\left \| (\uu,\ww) \right \|_{L^{2}(\mathbb{R}^{3})} = O(t^{-\alpha }) $ , for some $\alpha \geq 0$. Then, we have
	\begin{equation}\label{est grad}
		\left \| (\nabla \uu,\nabla \ww) \right \|_{L^{2}(\mathbb{R}^{3})} = O(t^{-\alpha - \frac{1}{2} }) .
	\end{equation}
\end{Proposition}
\begin{proof}
Define, for simplicity, $\zz(\cdot,t) := (\uu,\ww)(\cdot,t)$ and $D^{m}\zz =(D^{m}\uu,D^{m}\ww)$ , for each $m \geq 0$ integer. In order to show \eqref{est grad}, we observe, after a few computations, that
\begin{equation}\label{1 estimativa}
	\begin{split}
		&(t-t_{0})^{\widetilde{\gamma }} \left \| \zz(\cdot,t) \right \|_{L^{2}(\mathbb{R}^{3})}^{2} + 2\lambda \int_{t_{0}}^{t}(\tau -t_{0})^{\widetilde{\gamma }}\left \| D\zz(\cdot,\tau ) \right \|_{L^{2}(\mathbb{R}^{3})}^{2}d\tau \\
		&\leq \widetilde{\gamma } \int_{t_{0}}^{t}(\tau -t_{0})^{\widetilde{\gamma }-1} \left \| \zz (\cdot,\tau ) \right \|_{L^{2}(\mathbb{R}^{3})}^{2}d\tau \\
		&\leq C \int_{t_{0}}^{t}(\tau -t_{0})^{\widetilde{\gamma }-1} \tau^{-2\alpha }d\tau \\
		&\leq C (t -t_{0})^{-2\alpha + \widetilde{\gamma }} ,
	\end{split}
\end{equation}
	for all $\widetilde{\gamma } > 2 \alpha $. By using Lemma \ref{GSN}, we get
\begin{equation}\label{2 estimativa}
	\begin{split}
		&(t-t_{0})^{\widetilde{\gamma }+1} \left \| D\zz(\cdot,t) \right \|_{L^{2}(\mathbb{R}^{3})}^{2} + 2\lambda \int_{t_{0}}^{t}(\tau -t_{0})^{\widetilde{\gamma }+1}\left \|D^{2}\zz(\cdot,\tau ) \right \|_{L^{2}(\mathbb{R}^{3})}^{2}d\tau \\
		&\leq (\widetilde{\gamma }+1) \int_{t_{0}}^{t}(\tau -t_{0})^{\widetilde{\gamma }} \left \| D\zz(\cdot,\tau ) \right \|_{L^{2}(\mathbb{R}^{3})}^{2}d\tau\\
		&\; \; \; + C \int_{t_{0}}^{t}(\tau -t_{0})^{\widetilde{\gamma }+1}  \left \| \zz(\cdot,\tau ) \right \|_{L^{2}(\mathbb{R}^{3})}^{\frac{1}{2}}  \left \| D\zz(\cdot,\tau ) \right \|_{L^{2}(\mathbb{R}^{3})}^{\frac{1}{2}}  \left \| D^{2}\zz(\cdot,\tau ) \right \|_{L^{2}(\mathbb{R}^{3})}^{2} d\tau\\
		&\leq (\widetilde{\gamma }+1) \int_{t_{0}}^{t} (\tau -t_{0})^{\widetilde{\gamma }} \left \| D\zz(\cdot,\tau ) \right \|_{L^{2}(\mathbb{R}^{3})}^{2}d\tau\\
		&\; \; \; + C \int_{t_{0}}^{t}(\tau -t_{0})^{\widetilde{\gamma }+1} \left \| D\zz(\cdot,\tau ) \right \|_{L^{2}(\mathbb{R}^{3})}^{\frac{1}{2}}\left \| D^{2}\zz(\cdot,\tau ) \right \|_{L^{2}(\mathbb{R}^{3})}^{2}d\tau
        \end{split}
\end{equation}	
	
	For some $C>0$. By Lemma \ref{estimativa grad un medio} , there exists a $t_{0}$ sufficiently large such that ,  
	\begin{equation}\label{3 estimativa}
	C^{2}\left \| (D\uu,D\ww)(\cdot,t) \right \|_{L^{2}(\mathbb{R}^{3})}\leq \lambda ^{2} \quad  ;   \quad{\text{for all}} \quad t \geq t_{0} ,
	\end{equation}
combining \eqref{1 estimativa}, \eqref{2 estimativa} and \eqref{3 estimativa}, we have
	\begin{equation}\label{4 estimativa}
		\begin{split}
		(t-t_{0})^{\widetilde{\gamma }+1} \left \| D\zz(\cdot,t) \right \|_{L^{2}(\mathbb{R}^{3})}^{2}&+\lambda \int_{t_{0}}^{t}(\tau -t_{0})^{\widetilde{\gamma }+1} \left \| D^{2}\zz(\cdot,\tau ) \right \|_{L^{2}(\mathbb{R}^{3})}^{2}d\tau \\
		&\leq (\widetilde{\gamma }+1)  \int_{t_{0}}^{t} (\tau -t_{0})^{\widetilde{\gamma }} \left \| D\zz(\cdot,\tau ) \right \|_{L^{2}(\mathbb{R}^{3})}^{2}d\tau \\
		&\leq \widetilde{C } (t-t_{0})^{-2\alpha +\widetilde{\gamma } }
		\end{split}
	\end{equation}
	Applying \eqref{4 estimativa}, we get
	\begin{align*}
		\left \| (D\uu,D\ww)(\cdot,t) \right \|_{L^{2}(\mathbb{R}^{3})}
		\leq \widetilde{C } t^{-\alpha - \frac{1}{2} } .
	\end{align*}
Thus we complete the proof of \eqref{est grad}.
	
\end{proof}

\subsection{Linear operator related to \eqref{micropolar}} 
In this section we described some properties regarding the linear problem associated to \eqref{micropolar}. All the results presented here are useful to provide the $ \dot{H}^m $ estimates established in the section 4. 

We first consider the linear system related to \eqref{micropolar}.

\begin{subequations}\label{eqn:linear-system}
	\begin{equation}\label{L_uu_equation}
		\mbox{\boldmath $\bar{u}$}_t
		\;=\;
		(\:\!\mu + \chi\:\!) \, \Delta \mbox{\boldmath $\bar{u}$}
		\,+\,2
		\,\chi \, \nabla \times {\bf \bar{w}},
	\end{equation}
	\begin{equation}\label{L_ww_equation}
		{\bf  \bar{w}}_t
		\;\!=\,
		\gamma \;\! \, \Delta {\bf {\bar{\ww}}}
		\,+\,\kappa\, \nabla \:\!(\:\!\nabla \cdot\;{\bf \bar{w}} \:\!)
		\,+ \,2\,\chi \, \nabla \times\, \mbox{\boldmath $\bar{u}$}
		\,-\, 4 \, \chi \, {\bf \bar{w}},
	\end{equation}
	\begin{equation}
		\nabla \cdot \mbox{\boldmath $\bar{u}$}(\cdot,t)  = 0.
	\end{equation}
\end{subequations}
We observe that the linear system 
\eqref{eqn:linear-system} has a solution $ \bar{\zz}(\cdot,t;t_0) \in C^0([t_0,\infty), L^2(\RR^3)) $, for each initial time $ t_0 \geq 0 $, that is,  
$\bar{\zz}(\cdot,t;t_0) = e^{\mathbb{A} \:\!(t-t_0)}  \zz(\cdot,t_0) 
$, where 
$$
\mathbb{A}\zz:=\left[
\begin{array}{ccc}
	(\mu+\chi)\Delta & 2\,\chi \;\!\nabla\wedge  \\
	2\,\chi \;\!\nabla\wedge &  \mathbb{L} - 4\chi\,\mbox{Id}_{3\times 3} 
	\\
\end{array}
\right]\!\!\;\!
\left[
\begin{array}{ccc}
	\uu \\
	\ww  
	\\
\end{array}
\right]\!\!\;\!,$$
and $ \mathbb{L\ww}:= \gamma\,\Delta \ww+ \kappa \nabla(\nabla\cdot \ww)  $ is the Lamé operator. 
In the frequency space, after taking the Fourier transform of system (\ref{eqn:linear-system}), one obtains

\begin{displaymath}
	\partial_t \widehat{\bar{\zz}} = M(\xi) \widehat{\bar{\zz}},
\end{displaymath}
where $M = M(\xi)$ is the matrix of symbols

\begin{equation}
	\label{eqn:matrix-symbols}
	M = \left( \begin{array}{ccc} - (\mu + \chi) |\xi|^2 Id_{3 \times 3} & i \chi R_3(\xi) 
		\\  i \chi    R_3 (\xi) & - ( \gamma |\xi|^2 + 2 \chi)  Id_{3 \times 3} -  \kappa\,\xi_i \xi_j  
	\end{array} \right).
\end{equation}
Here, $Id_{3 \times 3}$ denotes the $3\times 3$ identity matrix and $ i  R_3 (\xi)$ denotes the rotation matrix
\begin{displaymath}
	i  R_3 (\xi) =  i  \left( \begin{array}{ccc} 0 & \xi_3 & - \xi_2 \\ - \xi_3 & 0 & \xi_1 \\ \xi_2 & - \xi_1 & 0 \end{array} \right).
\end{displaymath}

In \cite{NichePerusato2020} (see lemma (2.5), p. 48), the authors proved the following estimate for the eigenvalues of $ M(\xi) $: 
\begin{displaymath}
	\lambda _{max} (M) \leq - C |\xi|^2, \qquad C =C (\mu, \chi, \gamma) > 0 , 
\end{displaymath} 
as long as $32 \chi (\mu + \chi + \gamma) > 1$. As a consequence, we immediately have the following upper bound for the semigroup $\bigl(e^{\mathbb{A}t}\bigr)_{t \geq 0}$ associated to \eqref{eqn:linear-system}.     
\begin{Lemma}\label{lema-linear_opertaror}
	Let $ \mathcal{G} \in L^2_\sigma(\mathbb{R}^3) \times {L}^2(\mathbb{R}^3) $. If  $32 \chi (\mu + \chi + \gamma) > 1$, then \footnote{Recalling that $ e^{\Delta \tau} $ denotes the heat semigroup.}
	\begin{equation}\label{eqn:linear_operator_eq}
		\|e^{\mathbb{A}t} \:\!\mathcal{G} \|_{L^2} \leq \left\|e^{c\,\Delta t} \:\!\mathcal{G} \right\|_{L^2},
	\end{equation} 
	for all $t \geq 0$, where $e^{\mathbb{A}t}$ is the semigroup generated by the linear operator $\mathbb{A}$ above. 
\end{Lemma} 
	\begin{proof}
		Using the Plancherel`s identity, we have $\|e^{\mathbb{A}t} \:\!\mathcal{G} \|_{L^2} \leq 
		\|e^{-C\,|\xi|^2} \hat{\mathcal{G}}\| = \|e^{c\,\Delta t} \:\!\mathcal{G} \|_{L^2}  $
	\end{proof}
As a byproduct of our investigation, we finish this section by showing a fundamental result regarding the $ \dot{H}^m $ estimates for the linear system \eqref{eqn:linear-system} associated to \eqref{micropolar}. This should be compared to the so-called inverse Wiegner's theorem for solutions to the Navier-Stokes equations, see \cite{Skalak2014}. The following estimate is needed to obtain Theorem \ref{thm_inequalities} and also might be of independent interest.   
\begin{Theorem}\label{inverse-wiegner}
	Let $ \lambda_{0}(\alpha) < \infty $, for some $ 0 < \alpha <  3/4 $. Then, one has $ \limsup_{t \to +\infty } t^{ \alpha+ \frac{m}{2}} \|e^{\mathbb{A}\,t} \zz_0 \|_{\dot{H}^m}  < C_{m,\alpha}\,\lambda_0(\alpha).$
\end{Theorem}
\begin{proof}
	We start with the following claim.
\begin{Claim}\label{Claim} Given any initial time $ t_0 \geq 0 $, one has
	
	\begin{equation}\label{eqn: claim_difference}
			 \limsup_{t \to \infty}\, t^{\alpha \,+\, \hat{\beta}}\, \|	\zz(\cdot,t) - e^{\mathbb{A}(\,t-t_0)} \zz(\cdot,t_0)\|_{L^2}   
			\leq C_{\,\beta,\,\hat{\beta},\, \lambda_0(\alpha)\,}, 
\end{equation}
where 
\begin{eqnarray*}
			C_{\,\beta,\,\hat{\beta},\, \lambda_0(\alpha)\,} =  \begin{cases}
					c_1\,2^{\alpha + \hat{\beta}}\,\lambda_0(\alpha)^2 + \tilde{c}_1\,\lambda_0(\alpha)^\beta, \quad 0 \leq \alpha <1/4  \\ 
					 c_2\,2^{\alpha + \hat{\beta}}\,\lambda_0(\alpha) + \tilde{c}_2\,\lambda_0(\alpha)^\beta, \quad 1/4 \leq \alpha <1/2\\ 
					c_3\,2^{\alpha + \hat{\beta}}\,\lambda_0(\alpha)^{1/2} + \tilde{c}_3\,\lambda_0(\alpha)^\beta, \quad 1/2 \leq \alpha <3/4, 
			\end{cases}   		
	\end{eqnarray*} 
\begin{eqnarray*}
	\hat{\beta} =  \begin{cases}
		\alpha +1/4, \quad 0 \leq \alpha <1/4 \text{\,\,and\,\,} \beta \geq \frac{8\alpha + 8}{4 \alpha + 3} \\ 
		1/4, \quad 1/4 \leq \alpha <1/2 \text{\,\,and\,\,} \beta \geq \frac{4\alpha + 8}{4 \alpha + 3}\\ 
		0, \quad 1/2 \leq \alpha <3/4 \text{\,\,and\,\,} \beta \geq \frac{4\alpha + 7}{4 \alpha + 3}, 
	\end{cases}   		
\end{eqnarray*}
and $ c_i $,\,$\tilde{c}_i $ are constants.
\end{Claim}
Before proving the claim above\footnote{For the proof of claim, see the appendix section.}, we  obtain the Theorem \ref{inverse-wiegner} by using the following argument (adapted from \cite{BPZ}). By Theorem \ref{thm_L2} we have that $ \lambda_{0}(0) = 0 $, so we can apply the above estimate \eqref{eqn: claim_difference} for $ \alpha = 0 $ which immediately  leads us to $ t^{1/4} \|\zz(\cdot,t) - e^{\mathbb{A}(t-t_0)} \zz(\cdot, t_0) \| \to 0 $ as $ t \to \infty $. Now, we just have to notice that 
\begin{equation}\label{step1}
	\begin{split}
		\|e^{\mathbb{A}(t-t_0)} \zz(\cdot, t_0) \| \leq \|\zz(\cdot,t) \| + \|\zz(\cdot,t) - e^{\mathbb{A}(t-t_0)} \zz(\cdot, t_0) \| \\ 
		= O(t^{-\alpha})  + O(t^{-1/4}) = O(t^{-\gamma}), 
	\end{split}
\end{equation}
where $ \gamma = \min\{\alpha, 1/4 \} $. Therefore, if $ \alpha \leq 1/4 $, then Theorem \ref{inverse-wiegner} is already shown. The remaining cases for $\alpha $ will be considered as follows. When $ 1/4 <\alpha \leq 1/2 $,  by \eqref{step1}, we have $ \|\zz(\cdot, t) \| = O(t^{-1/4}) $, so that, by \eqref{eqn: claim_difference}, we have $\|\zz(\cdot,t) - e^{\mathbb{A}(t-t_0)} \zz(\cdot, t_0) \| = O(t^{-1/2})$. As before, we get    
\begin{equation*}
	\begin{split}
		\|e^{\mathbb{A}(t-t_0)} \zz(\cdot, t_0) \|_{L^2} \leq \|\zz(\cdot,t) \|_{L^2} + \|\zz(\cdot,t) - e^{\mathbb{A}(t-t_0)} \zz(\cdot, t_0) \|_{L^2} \\ 
		= O(t^{-\alpha})  + O(t^{-1/2}) = O(t^{-\alpha})
	\end{split}
\end{equation*}
and we are done in this case since $ 1/4 <\alpha \leq 1/2 $. Similarly, we can complete the proof for the case $ \alpha <3/4 $.  Adapting the argument presented in \cite{MR4021907} and further developed in \cite{MR4546674}, we just have to note that for all $ m \geq 0 $,
$$ \|e^{\mathbb{A}(t-t_0)} \zz(\cdot, t_0) \|_{L^2} = O(t^{-\alpha}) \implies \|D^m\,e^{\mathbb{A}(t-t_0)} \zz(\cdot, t_0) \|_{L^2} = O(t^{-\alpha - m/2})  \,.$$
 
For convenience, the proof of Claim \ref{Claim} will be given in Appendix.   
\end{proof}

\begin{Remark}
Due to this last result, the same interesting topological properties (obtained by L. Brandolese, C. Perusato and P. Zingano in \cite{BPZ}) hold for solutions to \eqref{micropolar}.
\end{Remark}

We are now in position to prove the main results.
\section{Proof of $L^2$ results}
\subsection{Proof of Theorem \ref{thm_L2}}
Let $ t _0>t_* $, where $ t_* $ is the regularity time (see \eqref{regularity_time}). First, we will prove the result for the micro-rotational field $\ww(\cdot,t)$. Defining $Z(\cdot,t)=e^{4\chi t}\ww(\cdot,t)$ and applying Duhamel's Principle, for $Z(\cdot,t)$ we obtain, after rewritting for $\ww(\cdot,t)$, that 
\begin{align*}
	\ww(\cdot,t) &=\: e^{\gamma \Delta  (t-t_{0})}e^{-4\chi (t-t_{0}) }\ww(\cdot,t_{0}) \:-\: \int_{t_{0}}^{t}e^{\gamma \Delta  (t-\tau )}e^{-4\chi (t-\tau) }(\uu \cdot \nabla \ww )(\cdot,\tau)d \tau\\
	&\:\:\:\:\:+\: \kappa \int_{t_{0}}^{t}e^{\gamma \Delta  (t-\tau )}e^{-4\chi (t-\tau) } \nabla(\nabla \cdot \ww )(\cdot,\tau)d \tau \\
	&\:\:\:\:\:+\: 2\chi \int_{t_{0}}^{t}e^{\gamma \Delta  (t-\tau )}e^{-4\chi (t-\tau) }(\nabla \times \uu  )(\cdot,\tau)d \tau .
\end{align*}
Hence,
\begin{align*}
	t^{\frac{1}{2}}\left \| \ww(\cdot,t) \right \|_{L^{2}(\mathbb{R}^{3})} & \leq \: t^{\frac{1}{2}}e^{-4\chi (t-t_{0}) }\left \| e^{\gamma \Delta  (t-t_{0})}\ww(\cdot,t_{0}) \right \|_{L^{2}(\mathbb{R}^{3})} \\
	&\:\:\:+\: t^{\frac{1}{2}}\int_{t_{0}}^{t}e^{-4\chi (t-\tau) }\left \| e^{\gamma \Delta  (t-\tau )}(\uu \cdot \nabla \ww )(\cdot,\tau) \right \|_{L^{2}(\mathbb{R}^{3})}d \tau\\
	&\:\:\:+\:\kappa t^{\frac{1}{2}} \int_{t_{0}}^{t}e^{-4\chi (t-\tau) } \left \| e^{\gamma \Delta  (t-\tau )}\nabla(\nabla \cdot \ww )(\cdot,\tau) \right \|_{L^{2}(\mathbb{R}^{3})}d \tau \\
	&\:\:\:+\: 2\chi t^{\frac{1}{2}} \int_{t_{0}}^{t}e^{-4\chi (t-\tau) }\left \| e^{\gamma \Delta  (t-\tau )}(\nabla \times \uu  )(\cdot,\tau) \right \|_{L^{2}(\mathbb{R}^{3})}d \tau  \\
	&  = \: I \:+\: II \:+\: III \:+\: IV
\end{align*}

It is clear that $t^{\frac{1}{2}}e^{-4\chi (t-t_{0}) } \left \| e^{\gamma \Delta  (t-t_{0})} \ww(\cdot,t_{0}) \right \|_{L^{2}(\mathbb{R}^{3})} \: \rightarrow  \: 0$ \,\, as\,\, $t \to +\infty$.\\

For $II$ , we use Lemma \ref{lema D alfa} and Lemma \ref{estimativa grad un medio}, we have
\begin{align*}
	&t^{\frac{1}{2}}\int_{t_{0}}^{t}e^{-4\chi (t-\tau) }\left \| e^{\gamma \Delta  (t-\tau )}(\uu \cdot \nabla \ww )(\cdot,\tau) \right \|_{L^{2}(\mathbb{R}^{3})}d \tau \\
	&\leq \:  K \gamma^{-\frac{3}{4}} t^{\frac{1}{2}} \int_{t_{0}}^{t}e^{-4\chi (t-\tau) }\left \| (\uu \cdot \nabla \ww )(\cdot,\tau) \right \|_{L^{1}(\mathbb{R}^{3})} (t-\tau )^{-\frac{3}{4}} d \tau \\
	&\leq \:  K \gamma^{-\frac{3}{4}} t^{\frac{1}{2}} \int_{t_{0}}^{t}e^{-4\chi (t-\tau) } \left \| \uu (\cdot,\tau) \right \|_{L^{2}(\mathbb{R}^{3})} \left \| D\ww (\cdot,\tau) \right \|_{L^{2}(\mathbb{R}^{3})}(t-\tau )^{-\frac{3}{4}} d \tau\\
	&\leq \:  K \gamma^{-\frac{3}{4}} \left \| (\uu ,\ww)(\cdot, t_{0} ) \right \|_{L^{2}(\mathbb{R}^{3})} t^{\frac{1}{2}} \int_{t_{0}}^{t}e^{-4\chi (t-\tau)} \left \| D\ww (\cdot,\tau) \right \|_{L^{2}(\mathbb{R}^{3})} (t-\tau )^{-\frac{3}{4}} d \tau \\
	&\leq \:  K \gamma^{-\frac{3}{4}} \left \| (\uu ,\ww)(\cdot, t_{0} ) \right \|_{L^{2}(\mathbb{R}^{3})} \epsilon   t^{\frac{1}{2}} \int_{t_{0}}^{t}e^{-4\chi (t-\tau)}  \tau^{-\frac{1}{2}}    (t-\tau )^{-\frac{3}{4}} d \tau \\
	&\leq \: K\gamma ^{-\frac{3}{4}}\epsilon \left \| (\uu,\ww)(\cdot,t_{0}) \right \|_{L^{2}(\mathbb{R}^{3})} \left ( t^{\frac{1}{4}} e^{- 2\chi t} \:+\: (4\chi)^{-\frac{1}{4}}\Gamma \left ( \frac{1}{4} \right ) \right ) .
\end{align*}
Therefore
\begin{equation}\label{2 w t}
	\lim _{t\rightarrow +\infty }t^{\frac{1}{2}}\int_{t_{0}}^{t}e^{-4\chi (t-\tau) }\left \| e^{\gamma \Delta  (t-\tau )}(\uu \cdot \nabla \ww )(\cdot,\tau) \right \|_{L^{2}(\mathbb{R}^{3})}d \tau = 0 .
\end{equation}\\
For $III$ , we use Lemma \ref{lema D alfa} and Lemma \ref{estimativa grad un medio}, we have
\begin{align*}
	&\kappa t^{\frac{1}{2}}\int_{t_{0}}^{t}e^{-4\chi (t-\tau) } \left \| e^{\gamma \Delta  (t-\tau )}\nabla(\nabla \cdot \ww )(\cdot,\tau) \right \|_{L^{2}(\mathbb{R}^{3})}d \tau  \\
	&\leq \: \kappa K \gamma^{-\frac{1}{2}} t^{\frac{1}{2}}\int_{t_{0}}^{t}e^{-4\chi (t-\tau) } \left \| D\ww(\cdot,\tau) \right \|_{L^{2}(\mathbb{R}^{3})}  (t-\tau)^{-\frac{1}{2}}d \tau\\
	&\leq\: \kappa K \gamma^{-\frac{1}{2}} \epsilon \: t^{\frac{1}{2}} \int_{t_{0}}^{t}e^{-4\chi (t-\tau) }\tau^{-\frac{1}{2}}(t-\tau)^{-\frac{1}{2}}d \tau \\
	&\leq \: \kappa K \gamma ^{-\frac{1}{2}}\epsilon \left [ t^{\frac{1}{2} }e^{-2\chi t} \:+\: (4\chi )^{-\frac{1}{2}}\sqrt{\pi } \right ] .
\end{align*}

So,
\begin{equation}\label{3 w t}
	\lim _{t\rightarrow +\infty }\kappa t^{\frac{1}{2}} \int_{t_{0}}^{t}e^{-4\chi (t-\tau) } \left \| e^{\gamma \Delta  (t-\tau )}\nabla(\nabla \cdot \ww )(\cdot,\tau) \right \|_{L^{2}(\mathbb{R}^{3})}d \tau \:=\: 0 .
\end{equation}\\
For $IV$,
\begin{equation}\label{4 w t}
	\lim _{t\rightarrow +\infty } 2\chi \; t^{\frac{1}{2}} \int_{t_{0}}^{t}e^{-4\chi (t-\tau) }\left \| e^{\gamma \Delta  (t-\tau )}(\nabla \times \uu  )(\cdot,\tau) \right \|_{L^{2}(\mathbb{R}^{3})}d \tau \:=\: 0
\end{equation}
Combining \eqref{2 w t}, \eqref{3 w t} and \eqref{4 w t}, we get
\begin{equation*}
	\lim _{t\rightarrow +\infty } t^{\frac{1}{2}} \left \| \ww(\cdot,t) \right \|_{L^{2}(\mathbb{R}^{3})} \:=\: 0 .
\end{equation*}
Thus, we complete the proof of  \eqref{eqn_w_thm1}. Now, we will show that $\left \| \uu \right \|_{L^{2}(\mathbb{R}^{3})} \rightarrow 0 $. We start rewriting the first equation in \eqref{micropolar} as
\begin{align*}
	\uu_{t} \:=\: (\mu + \chi) \Delta \uu \:+\: {\bf F}(\cdot, \tau) 
\end{align*}
where
\begin{equation*}
	{\bf F}(\cdot, \tau)\:=\:-\uu \cdot \nabla{\uu}\:-\: \nabla{p}\:-\: \eta \left | \uu \right |^{\beta -1}\uu  \:+\: 2\chi \nabla \times \ww .
\end{equation*}
In other words,
\begin{align*}
	{\bf F}(\cdot, \tau) \:=\: \mathbb{P}_{h}\left [ - \uu \cdot \nabla{\uu} \:-\: \eta \left | \uu \right |^{\beta -1}\uu \:+\: 2\chi \nabla \times \ww \right ],
\end{align*}
where $\mathbb{P}_{h} $ denotes the Helmholtz - Leray projector. By Duhamel's principle, we get
\begin{equation*}
	\uu(\cdot,t)\:=\:e^{(\mu + \chi) \Delta (t-t_{0})}\uu(\cdot,t_{0}) \:+\: \int_{t_{0}}^{t}e^{(\mu + \chi) \Delta (t-\tau)} {\bf F}(\cdot,\tau)d\tau ,
\end{equation*}
Since the Heat Kernel conmutes with the Helmholtz projector, we have
\begin{align*}
	\left \| \uu(\cdot,t) \right \|_{L^{2}(\mathbb{R}^{3})} & \:\leq \: \left \| e^{(\mu + \chi) \Delta (t-t_{0})}\uu(\cdot,t_{0}) \right \|_{L^{2}(\mathbb{R}^{3})} \\
	&\; \; \;\:+\: \int_{t_{0}}^{t} \left \| e^{(\mu +\chi) \Delta (t-\tau)} (\uu \cdot \nabla \uu)(\cdot,\tau) \right \|_{L^{2}(\mathbb{R}^{3})}d\tau \\
	&\; \; \;\:+\: \eta \int_{t_{0}}^{t} \left \| e^{(\mu +\chi) \Delta (t-\tau)} ( \left | \uu \right |^{\beta -1}\uu )(\cdot,\tau) \right \|_{L^{2}(\mathbb{R}^{3})}d\tau \\
	&\; \; \;\:+\: 2\chi \int_{t_{0}}^{t} \left \| e^{(\mu +\chi) \Delta (t-\tau)} ( \nabla \times \ww)(\cdot,\tau) \right \|_{L^{2}(\mathbb{R}^{3})}d\tau \\
	&\;  = \: I \:+\: II \:+\: III \:+\: IV .
\end{align*}
It is clear that $\left \| e^{(\mu +\chi)  \Delta (t-t_{0})}\uu(\cdot,t_{0}) \right \|_{L^{2}(\mathbb{R}^{3})} \rightarrow 0$ \quad  as \quad  $t \rightarrow +\infty $.\\

For $II$, we use again \ref{lema D alfa} and \ref{estimativa grad un medio}, we get
\begin{align*}
	&\int_{t_{0}}^{t} \left \| e^{(\mu +\chi) \Delta (t-\tau)} \uu(\cdot,\tau)\cdot \nabla{\uu}(\cdot,\tau)  \right \|_{L^{2}(\mathbb{R}^{3})}d\tau \\
	&\leq K (\mu+\chi)  ^{-\frac{3}{4}}\int_{t_{0}}^{t}(t-\tau)^{-\frac{3}{4}}\left \| \uu(\cdot,\tau)\cdot \nabla{\uu}(\cdot,\tau)  \right \|_{L^{1}(\mathbb{R}^{3})}d\tau  \\
	& \leq \: K (\mu+ \chi )^{-\frac{3}{4}}\sqrt{3} \int_{t_{0}}^{t} (t-\tau)^{-\frac{3}{4}}\left \| \uu(\cdot,\tau ) \right \|_{L^{2}(\mathbb{R}^{3})}\left \| D\uu(\cdot,\tau ) \right \|_{L^{2}(\mathbb{R}^{3})}d\tau\\
	&\leq \: K (\mu+ \chi )^{-\frac{3}{4}} \left \| (\uu,\ww)(\cdot,t_{0} ) \right \|_{L^{2}(\mathbb{R}^{3})} \int_{t_{0}}^{t} (t-\tau)^{-\frac{3}{4}}  \left \| (D\uu,D\ww)(\cdot,\tau ) \right \|_{L^{2}(\mathbb{R}^{3})}d\tau\\
	&\leq \: K(\mu+ \chi )^{-\frac{3}{4}} \left \| (\uu,\ww)(\cdot,t_{0} ) \right \|_{L^{2}(\mathbb{R}^{3})} \epsilon \int_{t_{0}}^{t} (t-\tau)^{-\frac{3}{4}}\tau^{-\frac{1}{2}}d\tau \\
	&\leq \: K\: (\mu+ \chi )^{-\frac{3}{4}} \left \| (\uu,\ww)(\cdot,t_{0} ) \right \|_{L^{2}(\mathbb{R}^{3})} \: \epsilon \: t^{-\frac{1}{4}} .
\end{align*}
Therefore
\begin{equation}\label{2 u}
	\lim _{t \to +\infty } \int_{t_{0}}^{t} \left \| e^{(\mu+ \chi ) \Delta (t-\tau)} \uu(\cdot,\tau)\cdot \nabla{\uu}(\cdot,\tau)  \right \|_{L^{2}(\mathbb{R}^{3})}d\tau   \:=\:  0 .
\end{equation}
For $III$, applying lemma \ref{lema D alfa}, Gagliardo - Sobolev - Nirenberg inequality and lemma \ref{estimativa grad un medio}, we get
\begin{align*}
	&\int_{t_{0}}^{t}\left \| e^{(\mu+ \chi ) \Delta (t-\tau)} \eta \left | \uu \right |^{\beta -1}\uu(\cdot,\tau)  \right \|_{L^{2}(\mathbb{R}^{3})}d\tau \\ &=\: K(m)\: (\mu+ \chi )^{-\frac{3}{4}} 
	\int_{t_{0}}^{t}(t-\tau)^{-\frac{3}{4}} \left \| \eta \left | \uu \right |^{\beta -1}\uu (\cdot,\tau)  \right \|_{L^{1}(\mathbb{R}^{3})}d\tau \\
	&\: =\: K(m)\: (\mu+ \chi )^{-\frac{3}{4}} \eta \int_{t_{0}}^{t}(t-\tau)^{-\frac{3}{4}}\left \| \uu \right \|_{L^{\beta }(\mathbb{R}^{3})}^{\beta }d\tau\\
	&\leq K(m)\: (\mu+ \chi )^{-\frac{3}{4}}  \eta \int_{t_{0}}^{t}(t-\tau)^{-\frac{3}{4}}\left \| \uu \right \|_{L^{2}(\mathbb{R}^{3})}^{\frac{6-\beta}{2} }\left \| D\uu \right \|_{L^{2}(\mathbb{R}^{3})}^{\frac{3(\beta-2)}{2} }d\tau\\
	&\leq K(m)\: (\mu+ \chi )^{-\frac{3}{4}} \eta \: \left \| (\uu_{0},\ww_{0}) \right \|_{L^{2}(\mathbb{R}^{3})}^{\frac{6-\beta}{2} } \int_{t_{0}}^{t}(t-\tau)^{-\frac{3}{4}} \left ( \epsilon \tau ^{-\frac{1}{2}} \right )^{\frac{3(\beta-2)}{2}}d\tau \\
	&\leq K(m)\: (\mu+ \chi )^{-\frac{3}{4}} \eta \: C^{\frac{6-\beta}{2} }\: \epsilon^{\frac{3(\beta-2)}{2}} \int_{t_{0}}^{t}(t-\tau)^{-\frac{3}{4}} \tau^{-\frac{3(\beta-2)}{4}}d\tau .
\end{align*}
Now, we observe that
\begin{align*}
\int_{t_{0}}^{t}(t-\tau)^{-\frac{3}{4}} \tau^{-\frac{3(\beta-2)}{4}}d\tau = \left\{\begin{matrix}
	C(\beta ) \: t^{\frac{7 - 3\beta  }{4} } & ;  \forall \; \;  \frac{7}{3} \leq \beta < \frac{10}{3} \\ 
	C \: t^{-\frac{ 3  }{4} } \ln \left ( \frac{t}{2} \right ) \:+\: t^{\frac{7 - 3\beta  }{4} }  & ; \beta = \frac{10}{3} \\
	C(\beta ) \: \left [ t^{-\frac{3}{4} } + t^{\frac{7 - 3\beta  }{4} } \right ]  &  ; \forall \; \;  \beta > \frac{10}{3}
\end{matrix}\right.
\end{align*}

Therefore
\begin{equation}\label{3 u}
	\lim _{t \to +\infty } \int_{t_{0}}^{t} \left \| e^{(\mu+ \chi ) \Delta (t-\tau)} \eta \left | \uu \right |^{\beta -1}\uu(\cdot,\tau)   \right \|_{L^{2}(\mathbb{R}^{3})}d\tau   \:=\:  0
\end{equation}
for all $\beta \: \geq \: \frac{7}{3} $. Finally, for $IV$, by using Lemma \ref{lema D alfa} we may assume that $t_{0} $ is large enough such that
\begin{align*}
	\left \| \ww(\cdot,t)   \right \|_{L^{2}(\mathbb{R}^{3})}d\tau   \:<\:  \epsilon t^{\frac{1}{2}}  ,  \; \; \;\; \; \; \; \; \;  \forall t  >  t_{0}.
\end{align*}
So, we have
\begin{align*}
	&2\chi \int_{t_{0}}^{t}\left \| e^{(\mu +\chi )\Delta (t-\tau)} (\nabla\times \ww)(\cdot,\tau) \right \|_{L^{2}(\mathbb{R}^{3})}d\tau \\
	&\leq \: K(2\chi)(\mu +\chi )^{-\frac{1}{2}}\int_{t_{0}}^{t}(t-\tau)^{-\frac{1}{2}} \left \| \ww(\cdot,\tau) \right \|_{L^{2}(\mathbb{R}^{3})}d\tau \\
	&\leq \: K(2\chi)(\mu +\chi )^{-\frac{1}{2}}\epsilon \int_{t_{0}}^{t}(t-\tau)^{-\frac{1}{2}}\tau^{-\frac{1}{2}}d\tau \\
	&\leq K(2\chi)(\mu +\chi )^{-\frac{1}{2}} \: \epsilon \: C .
\end{align*}

Therefore
\begin{equation}\label{4 u}
	\lim _{t\rightarrow +\infty }2\chi \int_{t_{0}}^{t}\left \| e^{(\mu +\chi )\Delta (t-\tau)} (\nabla\times \ww)(\cdot,\tau) \right \|_{L^{2}(\mathbb{R}^{3})}d\tau \:=\: 0 .
\end{equation}\\
Combining \eqref{2 u}, \eqref{3 u} and \eqref{4 u}, we have
\begin{align*}
	\lim _{t\rightarrow +\infty }\left \| \uu(\cdot,t) \right \|_{L^{2}(\mathbb{R}^{3})} \:=\: 0 .
\end{align*}
which completes the proof of Theorem \ref{thm_L2} , for all $\beta \geq \frac{7}{3}$.

\section{Proof of $\dot{H}^m$ results}
A natural consequence of Theorem \ref{thm_L2} and Proposition \ref{proposition} is the following result.
\begin{Corollary}\label{corollary w}
		Let $(\uu,\ww)(\cdot,t)$ be any Leray solution of \eqref{micropolar} with $\| (\uu, \ww)(\cdot,t)\|_{L^2} = O(t^{-\alpha})$, for some $\alpha \geq 0$. Then, we have
		$\|\ww(\cdot,t) \|_{L^2} = O(t^{-\alpha -\frac{1}{2}})$.
\end{Corollary}

\subsection{Proof of Theorem \ref{thm_inequalities} }
Let $ \zz(\cdot,t) = (\uu,\ww)(\cdot, t) $ any Leray solution to \eqref{micropolar}. First, we note that by using  the corollary \ref{corollary w} we actually have  
\mbox{ } \\
\begin{equation*}
	\begin{split}
		\limsup_{t \to +\infty }t^\alpha\,\|\uu(\cdot,t) \|_{L^2} < \infty \\
		\,\,\, \Rightarrow \,\, \limsup_{t \to +\infty }t^\alpha\,\|\zz(\cdot,t) \|_{L^2} = \limsup_{t \to +\infty }t^\alpha\,\|\uu(\cdot,t) \|_{L^2} < \infty 
	\end{split}
\end{equation*}
\mbox{ } \\
Due to this, we will show the result for ${\lambda}_0:= \limsup_{t \to +\infty }t^\alpha \|\zz(\cdot,t) \|_{L^2} < \infty$.

Let $$ {\bf Q}  (\cdot,t): = \left[
\begin{array}{ccc}
	\mathbb{P}_h\,[-\,	(\uu \cdot \nabla)\uu - \eta|\uu|^{\beta - 1}\,\uu\,] \\
	-\,	(\uu \cdot \nabla)\ww  
\end{array}
\right]\!\!\;\!$$
We consider the following integral representation for $\zz(\cdot,t):=(\uu, \ww)(\cdot,t)$  
\begin{equation*}
	{\zz}(\cdot,t) = e^{\mathbb{A}(t-t_0)} \,\zz (\cdot,t_0)  + \int_{t_0}^{t} e^{\mathbb{A}(t - \tau) } {\bf Q}(\cdot,\tau) d\tau, \quad \forall t>t_0,
\end{equation*}
for $ t_0 > t_{*} $ (see \eqref{regularity_time}) where $ \mathbb{A} $ is the linear operator given by lemma \ref{lema-linear_opertaror} above.
Taking the operator $\nabla^{m}$, norm $\left \| \cdot \right \|_{L^{2}(\mathbb{R}^{3})}$ and multiplying by $t^{\alpha + \frac{m}{2}}$, one has\\
\begin{align*}
	t^{\alpha +\frac{m}{2}}\left \| \nabla^{m} \uu(\cdot,t) \right \|_{L^{2}(\mathbb{R}^{3})} & \leq \: t^{\alpha +\frac{m}{2}} \left \| \nabla^{m} e^{(\mu +\chi) \Delta  (t-t_{0})}\uu_{0} \right \|_{L^{2}(\mathbb{R}^{3})} \\
	&\:\:\:+\: t^{\alpha +\frac{m}{2}}\int_{t_{0}}^{t}\left \| \nabla^{m} e^{(\mu +\chi) \Delta  (t-\tau )} (\uu \cdot \nabla \uu )(\cdot,\tau) \right \|_{L^{2}(\mathbb{R}^{3})}d \tau\\
	&\:\:\:+\:  t^{\alpha +\frac{m}{2}} \int_{t_{0}}^{t} \left \| \nabla^{m} e^{(\mu +\chi) \Delta  (t-\tau )} \eta (\left | \uu \right |^{\beta -1} \uu) (\cdot,\tau) \right \|_{L^{2}(\mathbb{R}^{3})}d \tau \\
	&\:\:\:+\: 2\chi t^{\alpha +\frac{m}{2}} \int_{t_{0}}^{t}\left \| \nabla^{m}e^{(\mu +\chi) \Delta  (t-\tau )} (\nabla \times  \ww)  (\cdot,\tau) \right \|_{L^{2}(\mathbb{R}^{3})}d \tau \\
	& = \: I  \:+\:  II  \:+\:  III,  
\end{align*}
where we have used some basic properties for the Leray projector.

To estimate term I, we invoke  Theorem \ref{inverse-wiegner}, which provides a inverse Wiegner-type result. Therefore, 
\begin{equation}\label{1u sup}
	\lim _{t\rightarrow +\infty} t^{\alpha +\frac{m}{2}} \left \| \nabla^{m} e^{(\mu +\chi) \Delta  (t-t_{0})} \uu(\cdot,t_{0}) \right \|_{L^{2}(\mathbb{R}^{3})} \leq C_{m, \alpha} \lambda_0(\alpha).  
\end{equation}

For II, 

\begin{align*}
	&t^{\alpha +\frac{m}{2}} \int_{t_{0}}^{\frac{t}{2}} \left \| 
	\nabla^{m} e^{(\mu +\chi) \Delta  (t-\tau )} (\uu \cdot \nabla \uu )(\cdot,\tau) \right \|_{L^{2}(\mathbb{R}^{3})}d \tau \\
	&\leq \: (\mu +\chi)^{-\frac{m}{2}-\frac{5}{4}} \: t^{\alpha +\frac{m}{2}} \int_{t_{0}}^{\frac{t}{2}} (t-\tau )^{-\frac{m}{2}-\frac{5}{4}} \left \| \uu   \right \|_{L^{2}(\mathbb{R}^{3})}^{2} d\tau \\
	& \leq \: (\mu +\chi)^{-\frac{m}{2}-\frac{5}{4}} \: \left \| (\uu_{0},\ww_{0}) \right \|_{L^{2}(\mathbb{R}^{3})} \: t^{\alpha +\frac{m}{2}} \int_{t_{0}}^{\frac{t}{2}} (t-\tau )^{-\frac{m}{2}-\frac{5}{4}}  \left \| \uu   \right \|_{L^{2}(\mathbb{R}^{3})} d\tau \\
	& \leq \:C (\mu +\chi)^{-\frac{m}{2}-\frac{5}{4}} \left \| (\uu_{0},\ww_{0}) \right \|_{L^{2}(\mathbb{R}^{3})}  \: (\lambda _{0}(\alpha )+\epsilon ) \: \underbrace{t^{\alpha +\frac{m}{2}} \int_{t_{0}}^{\frac{t}{2}} (t-\tau )^{-\frac{m}{2}-\frac{5}{4}}  \tau^{-\alpha } d\tau}_{(*)} 
\end{align*}
Note that
\begin{align*}
(*)=\left\{\begin{matrix}
C_{1} (\alpha, m) \: t^{-\frac{1}{4}}	& \forall \: 0 \leq \: \alpha < \: 1 \\ 
C_{2} ( m) \: \left [ t^{ -\frac{1}{4}} \: \: \ln \left ( \frac{t}{2} \right ) \:+\:  t^{-\frac{1}{4}} \right ] 	& \alpha = 1 \\ 
C_{3} (\alpha, m) \: \left [ t^{\alpha -\frac{5}{4}} \:+\: t^{-\frac{1}{4}} \right ] 	& \forall \: 1 < \: \alpha < \: \frac{5}{4}
\end{matrix}\right.
\end{align*}


Therefore, we have
\begin{equation}\label{2u sup}
	\limsup _{t \to +\infty } t^{\alpha +\frac{m}{2}} \int_{t_{0}}^{t} \left \| \nabla ^{m} e^{(\mu +\chi) \Delta  (t-\tau )} (\uu \cdot \nabla \uu )(\cdot,\tau) \right \|_{L^{2}(\mathbb{R}^{3})}d \tau \: = \: 0 .
\end{equation}\\

For $III$, applying the Gagliardo–Nirenberg inequality, we get
\begin{align*}
	& t^{\alpha +\frac{m}{2}} \int_{t_{0}}^{t} \left \| \nabla^{m} e^{(\mu +\chi) \Delta  (t-\tau )} ( \eta \left | \uu \right | ^{\beta -1}\uu) (\cdot,\tau) \right \|_{L^{2}(\mathbb{R}^{3})}d \tau\\
	&\leq \: K(m) (\mu + \chi)^{-\frac{3}{4}-\frac{m}{2}} \eta \: t^{\alpha +\frac{m}{2}} \int_{t_{0}}^{t}(t-\tau)^{-\frac{3}{4}-\frac{m}{2}} \left \| \left | \uu \right |^{\beta -1}\uu (\cdot,\tau)  \right \|_{L^{1}(\mathbb{R}^{3})}d\tau \\
	&\:=\: K(m) (\mu + \chi)^{-\frac{3}{4}-\frac{m}{2}} \eta \:  t^{\alpha +\frac{m}{2}} \int_{t_{0}}^{t}(t-\tau)^{-\frac{3}{4}-\frac{m}{2}} \left \| \uu \right \|_{L^{\beta }(\mathbb{R}^{3})}^{\beta }d\tau\\
	&\leq K(m) (\mu + \chi)^{-\frac{3}{4}-\frac{m}{2}} \eta \: t^{\alpha +\frac{m}{2}} \int_{t_{0}}^{t}(t-\tau)^{-\frac{3}{4}-\frac{m}{2}}\left \| \uu \right \|_{L^{2}(\mathbb{R}^{3})}^{\frac{6-\beta}{2} }\left \| D\uu \right \|_{L^{2}(\mathbb{R}^{3})}^{\frac{3(\beta-2)}{2} }d\tau\\
	&\leq C(\mu,\chi,m,\eta) \: (\lambda _{0}(\alpha )+\epsilon)^{\frac{6-\beta}{2}}  t^{\alpha +\frac{m}{2}} \int_{t_{0}}^{t}(t-\tau)^{-\frac{3}{4}-\frac{m}{2}}  \tau^{- \alpha \left ( \frac{6-\beta}{2} \right )  } \tau^{ \left ( -\alpha - \frac{1}{2} \right )
		\left ( \frac{3(\beta-2)}{2} \right ) }d\tau\\
	&= \: C(\mu,\chi,m,\eta) \: (\lambda _{0}(\alpha )+\epsilon)^{\frac{6-\beta}{2}} \; \underbrace{t^{\alpha +\frac{m}{2}} \int_{t_{0}}^{t} \tau^{ \frac{6- \beta(4\alpha + 3)}{4}  } (t-\tau)^{-\frac{3}{4} -\frac{m}{2} } d\tau}_{(*)} .
\end{align*}
Note that 
\begin{align*}
(*) = \left\{\begin{matrix}
	C_{1}(\alpha,\beta,m) \: t^{\frac{4\alpha + 7 - \beta (4\alpha +3)  }{4} } & ; \; \; \; \; \; \forall \; \;  \frac{4\alpha + 7}{4\alpha +3} < \beta < \frac{10}{4\alpha + 3} \\ 
	C_{2}(m) \: t^{\alpha -\frac{3}{4} } \ln \left ( \frac{t}{2} \right ) \:+\: C_{2}(\alpha,m) \: t^{\frac{4\alpha  - 3  }{4} }  & ;  \; \; \beta = \frac{10}{4\alpha + 3} \\
	C_{3}(\alpha,\beta,m) \:  \left [ t^{\alpha -\frac{3}{4} } \:+\: t^{\frac{4\alpha + 7 - \beta (4\alpha +3)  }{4} } \right ]  &  ; \; \; \forall \; \;  \beta > \frac{10}{4\alpha + 3} .
\end{matrix}\right. 
\end{align*}

Therefore, we have
\begin{equation}\label{3u sup}
	\limsup _{t \to +\infty } t^{\alpha +\frac{m}{2}}\: \int_{t_{0}}^{t} \left \| \nabla^{m} e^{(\mu +\chi) \Delta  (t-\tau )} (\eta \left | \uu \right | ^{\beta -1}\uu) (\cdot,\tau) \right \|_{L^{2}(\mathbb{R}^{3})} d \tau \: 
	= \: 0   
\end{equation}
for all  $\beta \: > \: \frac{4 \alpha + 7}{4 \alpha + 3}$ and for all $0 \: \leq \: \alpha \: < \: \frac{3}{4} $ .\\


Therefore combining \eqref{1u sup}  ,   \eqref{2u sup}  and \eqref{3u sup}, we finally have
\begin{align*}
	\limsup _{t \to +\infty } t^{\alpha +\frac{m}{2}} \left \| \nabla ^{m}\uu \right \|_{L^{2}(\mathbb{R}^{3})} \leq C_{\alpha ,m} \lambda _{0}(\alpha ) .
\end{align*} \\
Thus we complete the proof of \eqref{inequality_z}. In order to prove assertion \eqref{inequality_w} of Theorem \ref{thm_inequalities}. As before, appliying the  Duhamel's principle, we get
\begin{align*}
	t^{\alpha +\frac{m+1}{2}}\left \| \ww(\cdot,t) \right \|_{L^{2}(\mathbb{R}^{3})} & \leq \: t^{\alpha +\frac{m+1}{2}}e^{-4\chi (t-t_{0}) }\left \|D^{m} e^{\gamma \Delta  (t-t_{0})} \ww(\cdot,t_{0}) \right \|_{L^{2}(\mathbb{R}^{3})} \\
	&\:\:\:+\: t^{\alpha +\frac{m+1}{2}}\int_{t_{0}}^{t}e^{-4\chi (t-\tau) }\left \| e^{\gamma \Delta  (t-\tau )} \nabla^{m}(\uu \cdot \nabla \ww )(\cdot,\tau) \right \|_{L^{2}(\mathbb{R}^{3})}d \tau\\
	&\:\:\:+\:\kappa t^{\alpha +\frac{m+1}{2}} \int_{t_{0}}^{t}e^{-4\chi (t-\tau) } \left \| e^{\gamma \Delta  (t-\tau )} \nabla^{m+2} \ww (\cdot,\tau) \right \|_{L^{2}(\mathbb{R}^{3})}d \tau \\
	&\:\:\:+\: 2\chi t^{\alpha +\frac{m+1}{2}} \int_{t_{0}}^{t}e^{-4\chi (t-\tau) }\left \| e^{\gamma \Delta  (t-\tau )} \nabla^{m+1} \uu  (\cdot,\tau) \right \|_{L^{2}(\mathbb{R}^{3})}d \tau \\
	&=\: I  \:+\: II \:+\:  III \:+\:  IV .
\end{align*}

First, it is clear that $t^{\alpha +\frac{m+1}{2}}e^{-4\chi (t-t_{0}) }\left \|D^{m} e^{\gamma \Delta  (t-t_{0})} \ww(\cdot,t_{0}) \right \|_{L^{2}(\mathbb{R}^{3})} \: \rightarrow \: 0$ as $t \rightarrow +\infty$.

For II,
\begin{align*}
	&t^{\alpha +\frac{m+1}{2}}\int_{t_{0}}^{t} e^{-4\chi (t-\tau) }\left \| e^{\gamma \Delta  (t-\tau )} D^{m}(\uu \cdot \nabla \ww )(\cdot,\tau) \right \|_{L^{2}(\mathbb{R}^{3})}d \tau\\
	&\leq \: K(m) \gamma ^{-\frac{3}{4}-\frac{m}{2}} \: t^{\alpha +\frac{m}{2}+ \frac{1}{2} }\int_{t_{0}}^{t} e^{-4\chi (t-\tau) } \left \| \uu \cdot \nabla \ww  \right \|_{L^{1}(\mathbb{R}^{3})}(t-\tau)^{-\frac{3}{4}-\frac{m}{2}} d \tau\\
	&\leq \: K(m) \gamma ^{-\frac{3}{4}-\frac{m}{2}} \: t^{\alpha +\frac{m}{2} +\frac{1}{2} }\int_{t_{0}}^{t} e^{-4\chi (t-\tau) } \left \| \uu   \right \|_{L^{2}(\mathbb{R}^{3})}\left \| \nabla \ww  \right \|_{L^{2}(\mathbb{R}^{3})} (t-\tau)^{-\frac{3}{4}-\frac{m}{2}}  d\tau \\
	&\leq \: K(m) \gamma ^{-\frac{3}{4}-\frac{m}{2}} \:(\lambda _{0}(\alpha )+\epsilon )\: \underbrace{t^{\alpha +\frac{m}{2}+ \frac{1}{2} }\int_{t_{0}}^{t} e^{-4\chi (t-\tau) }  \tau ^{-2 \alpha - \frac{1}{2}} (t-\tau)^{-\frac{3}{4}-\frac{m}{2}}  d\tau}_{(*)} .  
\end{align*}

Note that
\begin{align*}
(*)= \left\{\begin{matrix} 
	C_{1}(\alpha,m)  \left ( t^{ \frac{1}{4}- \alpha} e^{-2\chi t} \:+\: (4\chi)^{-1} t^{-\alpha -\frac{3}{4}} \right )    & \;  ; \forall \; 0 \leq 
	\alpha < \frac{1}{4} \\
	C_{2}(m) \: \left [ e^{-2\chi t} \ln \left ( \frac{t}{2} \right ) \:+ \: (4\chi)^{-1} t^{-1} \right ]    & \alpha = \frac{1}{4} \\
	C_{3}(\alpha,m)  \left ( t^{\alpha - \frac{1}{4}} e^{-2\chi t} \:+\: (4\chi)^{-1} t^{-\alpha -\frac{3}{4}} \right )   & \;  ; \forall \; \alpha > \frac{1}{4}\\
\end{matrix}\right.
\end{align*}

Now, letting $t\rightarrow +\infty$, we have
\begin{equation}\label{2w sup}
	\limsup _{t \to +\infty } t^{\alpha +\frac{m+1}{2}}\int_{t_{0}}^{t}e^{-4\chi (t-\tau) }\left \| e^{\gamma \Delta  (t-\tau )} D^{m}(\uu \cdot \nabla \ww )(\cdot,\tau) \right \|_{L^{2}(\mathbb{R}^{3})}d \tau \:=\: 0 .
\end{equation}

For III,
\begin{align*}
	&\kappa \: t^{\alpha +\frac{m+1}{2}} \int_{t_{0}}^{t}e^{-4\chi (t-\tau) } \left \| e^{\gamma \Delta  (t-\tau )} D^{m+2} \ww (\cdot,\tau) \right \|_{L^{2}(\mathbb{R}^{3})}d \tau\\
	&\leq \: K(m)\: \kappa \: t^{\alpha +\frac{m+1}{2}} \int_{t_{0}}^{t}e^{-4\chi (t-\tau) } \left \| D^{m+2} \ww (\cdot,\tau) \right \|_{L^{2}(\mathbb{R}^{3})}d \tau \\
	&\leq \: K(m)\:(\lambda _{0}(\alpha )+\epsilon )\: \kappa \: t^{\alpha +\frac{m+1}{2}} \int_{t_{0}}^{t}e^{-4\chi (t-\tau) } \tau^{-\alpha -\frac{m+2}{2} }  d \tau \\
	&\leq \: K(m)\:(\lambda _{0}(\alpha )+\epsilon )\: \kappa \: t^{\alpha +\frac{m+1}{2}} \left [ \left ( \frac{1}{\alpha + \frac{m}{2}} \right ) \: t_{0}^{-\alpha -\frac{m}{2}}  \: e^{-2 \chi t }   \:+ \:   (4 \chi)^{-1 }  \left ( \frac{t}{2} \right )^{- \alpha -\frac{m+2}{2} }  \right ]\\
	&\leq \: K(\alpha, m)\:(\lambda _{0}(\alpha )+\epsilon )\: \kappa \: \left [ t^{\alpha +\frac{m+1}{2}} \: e^{-2\chi t } \:+\: t^{-\frac{1}{2}} (4\chi)^{-1 } \right ]
\end{align*}\\
so that we have
\begin{equation}\label{3w sup}
	\limsup _{t \to +\infty } \kappa \: t^{\alpha +\frac{m+1}{2}} \int_{t_{0}}^{t}e^{-4\chi (t-\tau) } \left \| e^{\gamma \Delta  (t-\tau )} D^{m+2} \ww (\cdot,\tau) \right \|_{L^{2}(\mathbb{R}^{3})}d \tau = 0.
\end{equation}\\
For IV, we get
\begin{align*}
	&2\chi t^{\alpha +\frac{m+1}{2}} \int_{t_{0}}^{t}e^{-4\chi (t-\tau) }\left \| e^{\gamma \Delta  (t-\tau )} D^{m+1} \uu  (\cdot,\tau) \right \|_{L^{2}(\mathbb{R}^{3})}d \tau \\
	&\leq \: 2\chi K(m) t^{\alpha +\frac{m+1}{2}} \int_{t_{0}}^{t}e^{-4\chi (t-\tau) }\left \| D^{m+1} \uu  (\cdot,\tau) \right \|_{L^{2}(\mathbb{R}^{3})}d \tau \\
	&\leq \: 2\chi K(m) \:(\lambda _{0}(\alpha )+\epsilon )\: t^{\alpha +\frac{m+1}{2}} \int_{t_{0}}^{t}e^{-4\chi (t-\tau) } \tau ^{-\alpha -\frac{m+1}{2}} d \tau \\
	&\leq \: 2\chi K(m) \: (\lambda _{0}(\alpha )+\epsilon )\:  t^{\alpha +\frac{m+1}{2}} \left [ C(\alpha,m) t_{0}^{-\alpha -\frac{m}{2} +\frac{1}{2}} e^{-2\chi t }  \:+\: (4 \chi)^{-1 } \left ( \frac{t}{2} \right ) ^{-\alpha  -\frac{m+1}{2}} \right ] \\
	&\leq \: 2\chi C(\alpha,m) \: (\lambda _{0}(\alpha )+\epsilon )\: \left [ e^{-2\chi t } \: t^{\alpha + \frac{m+1}{2}}   \: + \:  (4\chi)^{-1 }  \right ] .
\end{align*}\\
Letting $\epsilon \rightarrow 0$ and $t\rightarrow +\infty$ in this order, we have
\begin{equation}\label{4w sup}
	\begin{split}
	&\limsup _{t\rightarrow + \infty } t^{\alpha +\frac{m+1}{2}} \int_{t_{0}}^{t}e^{-4\chi (t-\tau) }\left \| e^{\gamma \Delta  (t-\tau )} D^{m+1} \uu  (\cdot,\tau) \right \|_{L^{2}(\mathbb{R}^{3})}d \tau \\ 
	&\leq \: C(\alpha,m) \lambda _{0}(\alpha ) .
	\end{split}
\end{equation}
	
Therefore, combining \eqref{2w sup} , \eqref{3w sup} and \eqref{4w sup} , we have
\begin{equation*}
	\limsup _{t\rightarrow +\infty } t^{\alpha +\frac{m+1}{2}} \left \| D^{m} \ww \right \|_{L^{2}(\mathbb{R}^{3})} \: \leq \: C(\alpha,m) \lambda _{0}(\alpha ),
\end{equation*}
which concludes the proof of \eqref{inequality_w}.
 
\newpage
\appendix
\section{Proof of Claim \ref{Claim}}

In this section we shall prove Claim \ref{Claim} which has an independent interest. We start with the following lemma for the associated linear system \eqref{eqn:linear-system}. 

\begin{Lemma}\label{thm_linear_aprox}
	Let $ \zz(\cdot,t) $ be a weak solution for system \eqref{micropolar} for all $ t>0 $.  If \,   $32 \chi (\mu + \chi + \gamma) > 1$, then, given any pair of initial times $ \tilde{t}_0 \geq t_0 \geq 0  $, one has
	\begin{equation}\label{eqn:Difference_estimate_z} 
		\begin{split}
			\|D^m \bar{\uu}(\cdot,t;{t_0}) - D^m\bar{\uu}(\cdot,t;\tilde{t}_0) \|_{L^2(\RR^3)}  \leq C\, (t-\tilde{t}_0)^{-\frac{5}{4} - \frac{m}{2}},
		\end{split}
	\end{equation} 
	\begin{equation}\label{eqn:Difference_estimate_w}
		\begin{split}
			\|D^m \bar{\ww}(\cdot,t;{t_0}) - D^m\bar{\ww}(\cdot,t;\tilde{t}_0) \|_{L^2(\RR^3)}  \leq C\, e^{-2\,\chi\, (t-\tilde{t}_0)}\,(t-\tilde{t}_0)^{-\frac{5}{4} - \frac{m}{2}},
		\end{split}
	\end{equation}
where $\bar{\uu}, \bar{\ww}$ denotes the solution to the linear system \eqref{eqn:linear-system} with initial data $ \bar{\uu}(\cdot,t_0) := \uu(\cdot,t_0) $ and $ \bar{\ww}(\cdot,t_0) := \ww(\cdot,t_0) $. To be more precise, we denote such solution by $\bar{\zz} (\cdot , t ; t_0) = ( \bar{\uu} (\cdot , t ; t_0) , \bar{\ww} (\cdot , t ; t_0))$
\end{Lemma}
\begin{proof}
	See e.g. Theorem 2.3 in \cite{BrazGuterresPerusatoZingano2024}. 
\end{proof}
Due to lemma \ref{thm_linear_aprox}  it is enough to prove the claim 
for times $t_0 $ larger than the regularity time  $t_*$. In fact, if the inequality \eqref{eqn: claim_difference} 
does hold for $t_0 > t_*$, then given any $\hat{t} \in [0 , t_*]$, Lemma
\ref{thm_linear_aprox} gives 
\begin{equation*}
	\begin{split}
		\limsup_{t \rightarrow \infty} t^{\alpha + \hat{\beta} } \|\Ez(\cdot,t;\hat{t}) \|_{L^2} \leq 
		\limsup_{t \rightarrow \infty} t^{\alpha + \hat{\beta} } \|\Ez(\cdot,t;t_0) \|_{L^2}  \\
		+ \limsup_{t \rightarrow \infty} t^{\alpha + \hat{\beta} } \underbrace{\|\Ez(\cdot,t;t_0) - \Ez(\cdot,t;\hat{t})  \|_{L^2}}_{= O(t^{-5/4})} \\ =    \limsup_{t \rightarrow \infty} t^{\alpha + \hat{\beta} } \|\Ez(\cdot,t;t_0) \|_{L^2} \leq C_{\,\beta,\,\hat{\beta},\, \lambda_0(\alpha)\,},
	\end{split}
\end{equation*}  
where, for convenience, we define $ \Ez(\cdot,t;t_0):= \zz(\cdot, t) - \bar{\zz}(\cdot, t; t_0) $. 
Now, by taking $ t_0 \geq t_\ast $ and using the integral representations for $ \zz(\cdot,t) $ and $ \bar{\zz} (\cdot , t ; t_0)  $, we get
\begin{align*}
	\Ez(\cdot,t;t_0) &\:=\: \zz(\cdot,t)  -  \bar{\zz}(\cdot,t;t_{0}) \\
	&\:=\: \zz(\cdot,t) - e^{\mathbb{A} (t-t_{0})} \zz(\cdot,t_{0}) \\
	&\:=\: \int_{t_{0}}^{t} e^{\mathbb{A} (t- \tau)} \begin{bmatrix}
		\mathbb{P}_{h}\left [ -\uu \cdot \nabla \uu - \eta \left | \uu \right |^{\beta -1}\uu \right ]\\ 
		-\uu \cdot \nabla \ww
	\end{bmatrix} d \tau 
\end{align*}
and, using Lemma \ref{lema-linear_opertaror}, this yields
\begin{equation}\label{error}
	\begin{split} 
	t^{\alpha +\hat{\beta} } & \left \|  \Ez(\cdot,t;t_{0}) \right \|_{L^{2}(\mathbb{R}^{3})} \\
	&\leq \: t^{\alpha +\hat{\beta}} \, \int_{t_{0}}^{t} \left \|  \, e^{c \Delta (t- \tau)} \,  (\uu \cdot \nabla \uu) \right \|_{L^{2}(\mathbb{R}^{3})} d\tau  \\
	&\; \; \; + \; t^{\alpha +\hat{\beta}} \,  \int_{t_{0}}^{t} \left \|  \, e^{c \Delta (t- \tau)} \, (\uu \cdot \nabla \ww) \right \|_{L^{2}(\mathbb{R}^{3})} d\tau  \\
	&\; \; \; + \; t^{\alpha +\hat{\beta}} \, \eta \int_{t_{0}}^{t} \left \|  \, e^{c \Delta  (t-\tau )} \, (  \left | \uu \right | ^{\beta -1}\uu) (\cdot,\tau) \right \|_{L^{2}(\mathbb{R}^{3})}d \tau. \\
	\end{split}
\end{equation}
We just need to examine each of the integrals above by splitting the interval $(t_0,t)$ into $(t_0, \mu(t)) \cup ( \mu(t), t) $, where $ \left.\mu(t):=\left(t+t_{0}\right) / 2\right) $. So, by applying the same idea used to obtain \eqref{2u sup} and \eqref{3u sup}, we get
\begin{align*}
	& t^{\alpha +\hat{\beta} } \int_{\frac{t}{2}}^{t} \left \|  e^{(\mu +\chi) \Delta  (t-\tau )} ( \uu \cdot \nabla \uu (\cdot,\tau) ) \right \|_{L^{2}(\mathbb{R}^{3})}d \tau\\
	&\leq \:C \,(\mu + \chi)^{-\frac{5}{4}} \: t^{\alpha + \hat{\beta} } \int_{\frac{t}{2}}^{t}  (t-\tau)^{-\frac{5}{4}}  \left \| \uu(\cdot,\tau) \right \|_{L^{2}(\mathbb{R}^{3})}^{2 } d\tau \\
	&\leq \: C\, (\mu + \chi)^{-\frac{5}{4}}  (\lambda _{0}(\alpha) + \epsilon )^{2} \: t^{\alpha + \hat{\beta} } \, \int_{\frac{t}{2}}^{t} (t-\tau)^{-\frac{5}{4}} \tau ^{-2 \alpha} d\tau  \\
	&\leq \: C\, (\mu + \chi)^{-\frac{5}{4}}  (\lambda _{0}(\alpha) + \epsilon )^{2} \, 2^{2\alpha -1},  
\end{align*}
  where we started by examining the integral on the interval $\left [ \frac{t}{2},t \right ]$. The first half needs to be studied in three cases depending on $\alpha$. 
  
  Case 1: $\alpha \in \left ( 0,\frac{1}{2} \right )$.
  \begin{align*}
  	& t^{\alpha +\hat{\beta}} \int_{t_{0}}^{\frac{t}{2}} \left \|  e^{(\mu +\chi) \Delta  (t-\tau )}  ( \uu \cdot \nabla \uu (\cdot,\tau) ) \right \|_{L^{2}(\mathbb{R}^{3})}d \tau\\
  	&\leq \: C\, (\mu + \chi)^{-\frac{5}{4}} \: t^{\alpha + \hat{\beta}} \int_{t_{0}}^{\frac{t}{2}}  (t-\tau)^{-\frac{5}{4}}  \left \| \uu(\cdot,\tau) \right \|_{L^{2}(\mathbb{R}^{3})}^{2 } d\tau \\
  	&\leq \: C(\mu , \chi) (\lambda _{0}(\alpha) + \epsilon)^{2} \: t^{\alpha + \hat{\beta}} \int_{t_{0}}^{\frac{t}{2}}  (t-\tau)^{-\frac{5}{4}} \tau ^{-2 \alpha} d\tau \\
  	&\leq \, C \; (\lambda _{0}(\alpha) + \epsilon)^{2} \: 2^{2\alpha + \frac{1}{4}} \: t^{-\alpha + \hat{\beta} - \frac{1}{4}} = \, C \; (\lambda _{0}(\alpha) + \epsilon)^{2}, 
  \end{align*}
where we take $ \hat{\beta} \,=\, \alpha + \frac{1}{4}$ in the last step. The second case is similar. 

Case 2: $\alpha \in \left ( \frac{1}{2},1 \right )$.
\begin{align*}
	& t^{\alpha +\hat{\beta}} \int_{t_{0}}^{\frac{t}{2}} \left \|  e^{(\mu +\chi) \Delta  (t-\tau )}  ( \uu \cdot \nabla \uu (\cdot,\tau) ) \right \|_{L^{2}(\mathbb{R}^{3})}d \tau\\
	&\leq \:C\, (\mu + \chi)^{-\frac{5}{4}} \: t^{\alpha + \hat{\beta}} \int_{t_{0}}^{\frac{t}{2}}  (t-\tau)^{-\frac{5}{4}}  \left \| \uu(\cdot,\tau) \right \|_{L^{2}(\mathbb{R}^{3})}^{2 } d\tau \\
	&\leq \:C(\mu ,\chi) \: t^{\alpha + \hat{\beta}} \int_{t_{0}}^{\frac{t}{2}}  (t-\tau)^{-\frac{5}{4}} \left \| \uu(\cdot,\tau) \right \|_{L^{2}(\mathbb{R}^{3})}  \left \| \zz(\cdot,t_{0}) \right \|_{L^{2}(\mathbb{R}^{3})} d\tau \\
	&\leq  \: C \; (\lambda _{0}(\alpha) + \epsilon) \: t^{\alpha + \hat{\beta}}  \left ( \frac{t}{2} \right )^{-\frac{5}{4}} \int_{t_{0}}^{\frac{t}{2}} \tau ^{- \alpha} d\tau 
	\leq C \; (\lambda _{0}(\alpha) + \epsilon) \:  \; t^{ \hat{\beta}  - \frac{1}{4}} 
	 \leq C \; (\lambda _{0}(\alpha) + \epsilon) ,    
\end{align*}
    where we take $\hat{\beta} \,=\, \frac{1}{4}$ in this case. 
    
    Case 3: $\alpha \in \left ( 1,\frac{5}{4} \right )$.
    \begin{align*}
    	& t^{\alpha +\hat{\beta}} \int_{t_{0}}^{\frac{t}{2}} \left \| e^{(\mu +\chi) \Delta  (t-\tau )}  (\uu \cdot \nabla \uu (\cdot,\tau) ) \right \|_{L^{2}(\mathbb{R}^{3})}d \tau\\
       	&\leq \: C(\mu , \chi) \: \left \| \zz(\cdot,t_{0}) \right \|_{L^{2}(\mathbb{R}^{3})} t^{\alpha + \hat{\beta}} \int_{t_{0}}^{\frac{t}{2}}  (t-\tau)^{-\frac{5}{4}} \left \| \uu(\cdot,\tau) \right \|_{L^{2}(\mathbb{R}^{3})}  d\tau \\
    	&\leq C\,(\lambda _{0}(\alpha) + \epsilon)  t^{\alpha + \hat{\beta}} \left ({t} \right )^{-\frac{5}{4}} \underbrace{\int_{t_{0}}^{\frac{t}{2}} \tau ^{- \alpha} d \tau}_{< \: \infty } 
    	\leq \: C\,(\lambda _{0}(\alpha) + \epsilon)\, t^{\alpha + \hat{\beta} - \frac{5}{4}} 
    	= \: C\,(\lambda _{0}(\alpha) + \epsilon)
    \end{align*}
    and here we finally take $ \hat{\beta} \,=\, \frac{5}{4} - \alpha$. The analysis of the remaining terms on \eqref{error} follows in a similar fashion.  

\bibliographystyle{plain}
\bibliography{PerusatoBiblio}{}

\end{document}